\documentclass[a4paper, 11pt]{amsart}

\usepackage{amsmath, latexsym, graphicx, amssymb, amscd, mathrsfs}
\usepackage[all]{xy}
\usepackage[british]{babel}
\usepackage{url,comment,array,caption}
\usepackage{appendix}

\usepackage{enumerate}

\usepackage{fullpage} 
\usepackage{color}

\definecolor{darkgreen}{rgb}{0,0.5,0} 
\usepackage[
        colorlinks, citecolor=darkgreen,
        backref,
        pdfauthor={Sara Checcoli, Francesco Veneziano, Evelina Viada; Appendix by Michael Stoll}, 
        pdftitle={The explicit Mordell conjecture for families of curves},
]{hyperref}
\usepackage[alphabetic,backrefs]{amsrefs} 

\renewcommand{\epsilon}{\varepsilon}
\renewcommand{\sharp}{\#}
\renewcommand{\emptyset}{\varnothing}

\newcommand{\abs}[1]{\left| #1 \right|}
\newcommand{\Q}{\mathbb{Q}}

\newcommand{\C}{\mathbb{C}}
\newcommand{\R}{\mathbb{R}}
\newcommand{\Z}{\mathbb{Z}}
\newcommand{\Qbar}{\overline{\mathbb{Q}}}
\renewcommand{\P}{\mathbb{P}}

\newcommand{\Ci}{\mathcal{C}}
\newcommand{\Di}{\mathcal{D}}

\newcommand{\SSS}{\mathscr{S}}

\newcommand{\Ccinque}{B_2}

\newcommand{\Cuno}{C_1}
\newcommand{\Cdue}{C_2}
\newcommand{\Ctre}{C_3}
\newcommand{\Cquattro}{B_1}

\newcommand{\Duno}{D_1}
\newcommand{\Ddue}{D_2}
\newcommand{\Dtre}{D_3}

\newcommand{\duno}{d_2}
\newcommand{\ddue}{d_1}

\newcommand{\cuno}{c_4}
\newcommand{\cdue}{c_9}
\newcommand{\ctre}{c_{10}}
\newcommand{\cquattro}{c_2}
\newcommand{\ccinque}{c_1}
\newcommand{\csei}{C_0}
\newcommand{\csette}{c_3}
\newcommand{\cotto}{c_6}
\newcommand{\cnove}{c_{5}}
\newcommand{\cdieci}{c_{7}}
\newcommand{\cundici}{c_{8}}

\newtheorem{thm}{Theorem}[section]
\newtheorem*{teo}{Theorem \ref{MAINT}'}

\newtheorem*{MC}{Mordell Conjecture}
\newtheorem*{MCL}{Mordell-Lang Conjecture}

\newtheorem{propo}[thm]{Proposition}
\newtheorem{lem}[thm]{Lemma}
\newtheorem{cor}[thm]{Corollary}

\newcommand{\PP}{\mathbb{P}}
\newcommand{\calC}{\mathcal{C}}
\newcommand{\calD}{\mathcal{D}}
\newcommand{\calW}{\mathcal{W}}
\newcommand{\tors}{{\operatorname{tors}}}
\newcommand{\injects}{\hookrightarrow}
\newcommand{\To}{\longrightarrow}

\numberwithin{equation}{section} 

\newtheorem{theorem}[thm]{Theorem}

\newtheorem{corollary}[thm]{Corollary}

\theoremstyle{definition}

\newtheorem{example}[thm]{Example}

\newtheorem{D}[thm]{Definition}

\theoremstyle{remark}
\newtheorem{remark}[thm]{Remark}

\title[]{The explicit Mordell Conjecture for families of curves\\ {\small (with an appendix by M.~Stoll)}}
\author{S.~Checcoli, F.~Veneziano, E.~Viada}
\begin{document}

 \begin{abstract}
In this article we prove the explicit Mordell Conjecture for large  families of curves. In addition, we introduce a method, of easy application,  to compute all rational points  on curves of quite general shape and  increasing genus.
The method bases on some explicit and sharp estimates for the height of such rational points, and the bounds are small enough to successfully implement a computer search.  As an evidence of the simplicity of its application, we present a variety of explicit examples and explain how to produce many others. In the appendix our method is compared in detail to the classical method of Manin-Demjanenko and the analysis of our explicit examples is carried to conclusion.
\end{abstract}

\maketitle

\section{Introduction}
The Diophantine problem of finding integral or rational solutions to a set of polynomial equations has been investigated since ancient times.  To this day there is no general method for finding such solutions and    the  techniques used to answer many fundamental  questions are deep and complex.
One of the leading principles in arithmetic geometry is that the geometric structure of an algebraic variety
 determines the arithmetic structure of the set of points over the rational numbers.

 A clear picture of  how the arithmetic mirrors the geometry for varieties is given by curves  defined over a number field $k$. The genus of the curve, a geometric invariant, distinguishes three qualitatively different behaviours for its rational points. For a curve of genus 0, either the set of $k$-rational point is empty or the curve is isomorphic to the projective line, whose $k$-rational points are infinitely many and well-understood. On the other hand, for genus at least 2 we have the:
\begin{MC}
A curve of genus at least 2 defined over a number field $k$ has only finitely many $k$-rational points.
\end{MC}
This is a very deep result, first conjectured by Mordell in \cite{MordellConj} and now known as Faltings Theorem after the ground-breaking proof in \cite{FaltingsTeo}.
In between these two extremes, there are the curves of genus 1.  They can be endowed with the structure of an abelian group and the set of $k$-rational points, when not empty, is a finitely generated group. This is a famous theorem of Mordell, later generalised by Weil to the case of abelian varieties.

\medskip

 Vojta in \cite{Vojta91} gave a new proof of the Mordell Conjecture and then Faltings, in \cite{FaltingsAnnals} and \cite{FaltingsBarsotti}, proved an analogous statement for rational points on subvarieties of abelian varieties, which generalises to points in a finitely generated subgroup $\Gamma$. Basing on these results, Hindry \cite{HindryAutour}  proved  the case of $\Gamma$ of finite rank, known as the Mordell-Lang Conjecture. This was later made quantitative by R\'emond  \cite{gaelLang}.
\begin{MCL}   Let $\Gamma$ be a subgroup of finite rank of an abelian variety $A$. Let $V\subseteq A$ be a  proper subvariety. Then the set $\Gamma\cap V$ is contained in a finite union of translates of proper abelian subvarieties by elements of $\Gamma$.
\end{MCL}

Unfortunately,  even for curves the different proofs of this theorem   are not effective,  in the sense  that they prove the finiteness of the desired set, but do not hint at how this set could be determined. One of the  challenges of the last century has been the search  for  effective methods, but  there is still no known general method for finding all the rational points on a curve.  The few available methods  work under special assumptions and explicit examples are mainly given for curves of genus $2$ or $3$ as discussed below.

The method of Chabauty-Coleman \cite{Chab}  and  \cite{Coleman} provides a bound on the number of rational points on  curves  defined over  a number field $k$ with Jacobian of $k$-rank  strictly smaller than the genus.
In   some   examples the estimate gives the exact number of rational points, so that, possibly in combination with \emph{ad hoc} descent arguments, one can  find the right number of points and list them.
See for example  Flynn \cite{Flynn} for one of the first  explicit applications of the Chabauty-Coleman method, Siksek \cite{Siksek} for investigations on  possible extensions of the method, McCallum and Poonen  \cite{MP}  and  Stoll \cite{Stoll} for  general surveys and also their  references for additional variations and applications of this method.
For curves of genus 2, one can find the rational points using an implementation by Stoll based on~\cite{StollMagma}*{Section~4.4} of the Chabauty-Coleman method combined with the Mordell-Weil Sieve in the Magma computational algebra system~\cite{Magma}; this works when the Mordell-Weil rank of the Jacobian is one and an explicit point of infinite order is known.

The  Manin-Demjanenko method (\cite{Demj},\cite{Manin}) is effective and applies to curves $\Ci$  defined over a number field $k$ that admit    $m$ morphisms $f_1,\ldots,f_m$ from $\mathcal{C}$ to an abelian variety $A$ all defined over $k$ and linearly independent modulo constants (in the sense that if $\sum_{i=1}^m n_i f_i$ is constant for some integers $n_i$, then $n_i=0$ for all $i$). If $m>\mathrm{rank}A(k)$, then $\mathcal{C}(k)$ is finite and may be found effectively.
However the method is far from being explicit in the sense that it does not give  the dependence of the height  of the rational points, neither  on the curve nor on the morphisms; this makes it difficult for applications. See Serre \cite{serreMWThm} for a description of the method and a few applications. In the papers of Kulesz \cite{KuleszApplicazioneMD},  Girard and Kulesz \cite{Kul2} and  Kulesz, Matera and Schost \cite{Kul3} this method has been used  to find all rational points on  some families of curves  of genus 2 (respectively 3) with  morphisms to special elliptic curves of rank 1 (respectively $\le 2$).   For instance, in \cite{KuleszApplicazioneMD} the curves  have Jacobian  with factors isogenous to $y^2=x^3+a^2 x$, with $a$ a square-free integer and such that the Mordell-Weil group has rank one.  We refer to Section~\ref{App:comparison} of the appendix for a more detailed discussion on the Manin-Demjanenko method, including a comparison with the results of this article.

We also mention that Viada   gave in \cite{ioannali} an effective method which is comparable with the setting of Manin-Demjanenko's result, although different in strategy. She obtains an effective height bound for the $k$-rational points on a transverse curve $\Ci\subseteq E^N$ where $E$ is an elliptic curve with $k$-rank  at most $N-1$. Also in this case the bounds are not at all explicit and there are no examples.

 A major shortcoming of these methods is that  in general the bounds for the height must be worked out case by case and  this is feasible in practice only when the equations of the curve  are of a very special shape.

\bigskip

In this article we  provide a  good explicit upper bound for the height of the  points in the intersection of a  curve of genus at least $2$ in $E^N$ with the union of all algebraic subgroups of dimension one, where $E$ is an elliptic curve without CM (Complex Multiplication), proving in this setting the explicit Mordell-Lang conjecture for points of rank one. With  some further  technical estimates, the method works also for the CM case.  Our method can be easily applied to find the rational points on curves of a fairly general shape and growing genus.  Moreover we present a variety of explicit examples,   given by curves of genus at least $2$ embedded in $E^2$, with $E$  without CM and $E(k)$ of rank one.   These are precisely the curves  whose Jacobian has  a factor isogenous to such an $E^2$. So the method can be easily applied to curves embedded in $E^2\times A$, where $A$ is an abelian variety. This is also the  first nontrivial setting, as the case of  $E(k)$ of rank zero can be easily treated (see Theorem \ref{prop.transv.A}  and Remark \ref{rem.rank0}).   Many explicit examples mentioned above can be covered by our method, but it also gives many new examples  in which, differently from all previous examples, the genus of the curves tends to infinity  (see also Appendix~\ref{appendix}, in particular Section~\ref{S:examples}).

Compared to the other effective methods  mentioned above, ours is  easy to apply because it provides a simple formula for the  bound  for the height of the rational points.
 Finally, in our settings the method of Chabauty-Coleman cannot be  directly applied, as the rank of the $k$-rational points of the ambient variety  is not smaller than its dimension. Our assumption is instead compatible with the Manin-Demjanenko setting.

The importance of the result is that the dependence of our bound for the height   is completely explicit both on the curve $\Ci$ and the elliptic curve $E$  and it can be directly computed from the coefficients  of the equations defining the curve.  More precisely, it depends explicitly on the coefficients of a Weierstrass equation for $E$ and on the degree and normalised height of $\Ci$.

  To give some evidence of the   power  of our method we carry out in this paper the following applications:
  \begin{itemize}
\item  the proof of  the explicit Mordell Conjecture for several families of curves,
\item  the list of all rational points for more than  $10^4$  explicit   curves.
\end{itemize}

\bigskip

To state our main theorem, we first fix the  setting (see  Section \ref{notaz} for more details). Let $E$ be an elliptic curve  given in the form
\begin{equation*}\label{uno}y^2=x^3+Ax+B.\end{equation*}
Via the given equation, we embed $E^N$ into $\P_2^N$ and via the Segre embedding in $\P_{3^N-1}$.

 The degree of a  curve   $\Ci\subseteq E^N$ is the degree of its image in $\P_{3^N-1}$  and $h_2(\Ci)$ is the normalised height of $\Ci$, which is defined in terms of the Chow form of the ideal of $\Ci$, as done  in  \cite{patrice}.  We let $\hat{h}$ be the N\'eron-Tate height  on $E^N$  (normalised as explained in Section \ref{heightsanddegrees}).

 We  finally define the rank for a point  of $E^N$ as the  ${\rm End }(E)$-rank of the ring generated by its coordinates or more in general:
\begin{D}
\label{rank}
 The rank of a point in an abelian variety $A$ is the minimal dimension  of   an algebraic subgroup containing the point.
\end{D}

 We can now state our main result:
 \begin{thm}\label{caso_E^N}
  Let $E$ be an elliptic curve without CM. Let $\Ci$ be  an irreducible curve of genus at least $2$ embedded in $E^N$. Then every point $P$ of rank at most one on $\Ci$ has N\'eron-Tate height bounded as
  \begin{align*} \hat{h}(P)\leq& 2\cdot  3^{N-2} N! \deg \Ci \left( \Cuno h_2(\Ci)(\deg\Ci) +\Cdue(E)(\deg\Ci)^2 + \Ctre(E)+\ccinque(E)+3^{N}\right)+\\ &+3^{N-2} (N-2)!h_2(\Ci)+ N\cquattro(E).
 \end{align*}

  Moreover if $N=2$
   \begin{equation*}\label{utile}
  \hat h(P)\leq \Cuno\cdot h_2(\Ci)\deg\Ci +\Cdue(E)(\deg\Ci)^2 +\Ctre(E)
\end{equation*}
    where
 \begin{align*}
    \Cuno&=72.251\\
    \Cdue(E)&=\Cuno\left(6.019+4\ccinque(E)\right)\\
    \Ctre(E)&=4\cquattro(E),
 \end{align*}
 and the constants $\ccinque(E)$ and $\cquattro(E)$ are defined in Table \ref{table1} and depend explicitly on the coefficients of $E$.
 \end{thm}

 \bigskip

Theorem \ref{caso_E^N}  is the combination of Theorem \ref{MAINT} proven in Section \ref{DIMTHMMAIN} and Theorem \ref{IS} proven in Section \ref{nuova}.

 We remark that if $E(k)$ has rank one then the set of $k$-rational points of $\Ci$ is contained in the set of points of rank one and so  it has height bounded as above.
 We underline that our method to bound the height of the rational points does not require the knowledge of a generator for $E(k)$ to work and that the bound we obtain is also independent on $k$. These aspects are rather important, specifically for applications.

 Our  search for effective and even explicit methods for the height of the $k$-rational points on curves started some years ago in the context of the Torsion Anomalous Conjecture (in short TAC), introduced by
 Bombieri, Masser and Zannier \cite{BMZ1}. It is well known that this very general conjecture on the finiteness of the maximal torsion anomalous varieties implies the Mordell-Lang Conjecture and that effective results in the context of the TAC carry over to effective cases of the Mordell-Lang Conjecture  (see  \cite{via15} for a survey).Several of the methods used in this field are based on a long-established strategy of using theorems of diophantine approximation to obtain results about the solutions to diophantine equations. This general approach goes back at least to Thue and Siegel and has been often applied with success in the field of unlikely intersections as well as in number theory in general  (see \cite{Zan12}   and references there for a nice overview). Despite much effort there are few effective methods in this context and ours is  probably the first explicit one in the setting of abelian varieties.

Our main theorem generalises and drastically improves a previous result obtained in \cite{ExpTAC}
where we considered only weak-transverse curves, i.e. curves not contained in any proper algebraic subgroup (see Definition \ref{deftrans}), a stronger assumption which does not cover all curves of genus $\ge2$ and we could only bound the height of the subset of points of  rank one which are also torsion anomalous. In spite of the more restrictive setting, the bounds obtained in \cite{ExpTAC} are much worse than the present ones and they are  beyond any hope of implementing them in any concrete case.

For instance, in this article, Theorem \ref{IS}, for weak-transverse curves in $E^N$ with $N\ge3$  we obtain
\begin{equation*}
  \hat h(P)\leq 4 (N-1) \Cuno h_2(\Ci)\deg\Ci +(N-1) \Cdue(E)(\deg\Ci)^2 +N^2\Ctre(E),
 \end{equation*}
while in \cite{ExpTAC} under the same hypothesis  we got
 \begin{equation*}
 \hat h(P)\leq \Cquattro(N)\cdot 2(N-1) \Cuno h_2(\Ci)(\deg\Ci)^{N-1} +\Ccinque(N)\cdot (N-1) \Cdue(E)(\deg\Ci)^N +N^2 \Ctre(E)
\end{equation*}
 where
$\Ccinque(N)\geq \Cquattro(N)\geq 10^{27} N^{N^2} (N!)^N$.
 Note that not only the constants here are linear instead of exponential in $N$, but also the exponents  of $\deg \Ci$ are now independent of $N$ and better already for $N=3$.

 By introducing new key elements in the proof, we go beyond what we could prove in \cite{ExpTAC}; this change in approach leads to improvements of the bounds crucial for the practical implementatation.

 More in detail, this is a sketch of the proof  of the main theorem given in  Sections \ref{deduz} and \ref{DIMTHMMAIN}. At first instance we avoid to  restrict ourselves to  the concept of torsion anomalous points as done in \cite{ExpTAC} and study all points of rank one.  To treat the case of a general $N$   we use a  geometric construction to reduce it to the case of $N=2$. In this case we do a typical  proof of diophantine approximation:  if $P$ is a point in $E^2$ of rank one, we construct a subgroup $H$ of dimension 1 such that the height and the degree of the translate $H+P$ are well controlled. To this aim  we use  some classical results of the geometry of numbers,  in a way that prevents the bounds from growing beyond the computational limits of a computer search. We then conclude the  proof using the Arithmetic B\'ezout Theorem, the Zhang inequality and an optimal choice of the parameters.

Another significant feature of our main theorem is that it can  easily be  applied to find the rational points on curves of quite general shape.  We present here some of these applications,  remarking that, for instance, any curve of genus at least $2$ in $E^2$ with $E(\Q)$  of rank one  is suitable for further examples of our method.

Let $E$ be an elliptic curve defined over $\overline{\Q}$. We write
\begin{equation}\label{due}
\begin{split}
y_1^2&=x_1^3+Ax_1+B\\
y_2^2&=x_2^3+Ax_2+B
\end{split}
\end{equation}
for the equations of $E^2$ in $\P_2^2$ using affine coordinates $(x_1,y_1)\times (x_2,y_2)$ and we embed $E^2$ in $\mathbb{P}_8$ via the  Segre embedding.

In Section \ref{sez:criterio}  we give a method  to construct several families of  curves in $E^2$ of growing genus and we show how to  compute   bounds for their degree and normalised height.  In Theorem \ref{caso_poly}  we prove a sharper version of the following  result.
\begin{thm}\label{polyintro}   Assume that $E$ is without CM, defined over a number field $k$ and that $E(k)$ has rank  one.
Let $\Ci$ be the projective closure of the curve given in $E^2$ by the additional equation $$p(x_1)=y_2,$$ with $p(X)\in k[X]$ a non-constant polynomial  of degree $n$. Then for $P\in \Ci(k)$ we have
\[\hat h(P)\leq 1301 (2n+3)^2\left(h_W(p)+\log n+2\cotto(E)+3.01+2\ccinque(E)\right)+4\cquattro(E)\]
where $h_W(p)=h_W(1:p_0:\ldots:p_n)$ is the height of the coefficients of $p(X)$, and the constants $\cotto(E)$, $\ccinque(E)$ and $\cquattro(E)$  are defined in Table \ref{table1}.
\end{thm}
We then consider  two specific families which have  particularly  small invariants. Clearly these are just examples and many similar others can be  given.
\begin{D}\label{defCn}\label{defDn}
 Let $\{\Ci_n\}_n$ be the family  of the projective closures of the curves in $E^2$ defined for $n\geq 1$ via the additional equation $$x_1^n=y_2.$$
 Let $\{\Di_n\}_n$ be the family  of the projective closures of the curves in $E^2$ defined for $n\geq 1$  via the additional equation $$\Phi_n(x_1)=y_2,$$ where $\Phi_n(x)$ is the $n$-th cyclotomic polynomial.
\end{D}
 In order to directly apply our main theorem we cut these curves on $E^2$, with $E$ varying in the set of  elliptic curves   over $\Q$ without CM and $\Q$-rank one. Several examples of such $E$ have been tabulated below and others can be easily found, for instance in Cremona's tables \cite{Cremona}.

We consider the following elliptic curves:
\begin{align*}
   E_1: y^2&=x^3+x-1,\\
   E_2: y^2&=x^3-26811x-7320618, \\
   E_3: y^2&=x^3-675243x-213578586,\\
   E_4: y^2&=x^3- 110038419x + 12067837188462,\\
   E_5: y^2&= x^3 - 2581990371 x - 50433763600098.
\end{align*}

These are five elliptic curves without CM and of rank one over $\Q$.
The curves $E_1,E_3,E_4,E_5$ are, respectively, the curves  248c.1,10014b.1, 360009g.1 and 360006h.2 of  \cite{Cremona}. The curve $E_2$ was considered by Silverman in \cite{Sil}, Example 3 and it does not appear in the Cremona Tables because its conductor is too big.  The curves $E_3, E_4$ and $E_5$ were chosen because they have generators of the Mordell-Weil group of large height. This choice  may speed-up the computations, but it is not necessary (see Section \ref{SecEx2} for more details)

A remarkable application of our  theorem is the following result, proven in Section \ref{SecEx2}.  If $E$ is an elliptic curve, we denote by $O$ its neutral element.
\begin{thm}\label{curveEsp}
For   the $79600$ curves $\Ci_n \subseteq E_i\times E_i$ with $1\leq n\leq 19900$ and $i=2,3,4,5$, we have
   $$\Ci_n(\Q)=\{O\times O\}.$$
      For   the $9900$  curves $\Ci_n \subseteq E_1\times E_1$ with $1\leq n\leq 9900$, we have $$\Ci_n(\Q)=\{O\times O, (1,\pm 1)\times(1,1)\}.$$
 For  the $5600$ curves $\Di_n \subseteq E_i\times E_i$ where  $1\leq n\leq 1400$ and $i=2,3, 4,5$ we have
   $$\Di_n(\Q)=\{O\times O\}.$$
   For  the  $400$ curves $\Di_n \subseteq E_1\times E_1$ with $1\leq n\leq 400$ we have
\begin{align*}
D_1(\Q)&= \{O\times O,(2,\pm 3)\times(1,1)\}\\
D_2(\Q)&=\{O\times O,(2,\pm 3)\times(2,3)\}\\
D_{3^k}(\Q)&=\{O\times O,(1,\pm 1)\times(2,3)\}\\
D_{47^k}(\Q)&=\{O\times O,(1,\pm 1)\times(13,47)\}\\
D_{p^k}(\Q)&=\{O\times O\} \text{ if $p\neq 3,47$ or $p=2$ and $k>1$}\\
D_6(\Q)&=\{O\times O,(1,\pm 1)\times(1,1),(2,\pm 3)\times(2,3)\}\\
D_n(\Q)&=\{O\times O,(1,\pm 1)\times(1,1)\}\text{ if $n\neq 6$ has at least two distinct prime factors.}
\end{align*}

\end{thm}

 For these curves  the bounds for the height of the rational points are very good  especially for the $\Ci_n$; in fact they are so good that we can carry out a fast  computer search and determine  all their rational points   for $n$ quite large.
The computations have been executed with the computer algebra system PARI/GP \cite{PARI} using an algorithm by K. Belabas discussed in Section \ref{SecEx2} based on a sieving method.

The computations for the  9900 curves $\Ci_n$ in $E_1^2$ took about 7 days. The  79600 curves $\Ci_n$ in $E_i^2, i=2,\dotsc,5$ took about $11$ days, while the computations on the  6000 curves $\Di_n$ took about three weeks. A single curve in this range takes between a few seconds and a few minutes, for example $\Ci_{1000}$ in $E_2^2$ takes about 6.8 seconds.

 In Appendix~\ref{appendix} M.~Stoll completes the study of the rational points on the families $\Ci_n$ and $\mathcal{D}_n$ for all $n$. More precisely, he proves that for $n$ large enough all rational points on the curves must be integral, by combining our upper bound for the height of the rational points with a lower bound obtained by studying the $\ell$-adic behaviour of points on the curve close to the origin, see Sections \ref{App:nonintegral} and~\ref{S:examples}. Thus our computations are required only for $n$ small. However the data above give an idea of the time needed to find the rational points on other curves with invariants similar to those considered in Theorem \ref{curveEsp}, even when the approach of the Appendix does not apply.

 For a few curves in which the bounds are particularly small, we first used a naive algorithm, which took about 6 weeks for each curve $\Ci_1$. Then we used a floating point algorithm suggested by J. Silverman:  for each $i=1,\ldots,5$ this algorithm  took about one week for the 10 curves $\Ci_n\in E_i^2$ with $1\le n\le 10$. The striking improvement in the running time is due to the idea of performing the computations after reducing modulo many primes; arithmetic operations in finite fields are much faster than exact arithmetic.  More details on how to construct suitable new examples are given in Section \ref{sez:criterio}.

\bigskip

The paper is organised as follows: Sections \ref{notaz} and \ref{altezze} contain the notations, definitions and some useful standard results.  In Section~\ref{deduz} we state Theorem~\ref{MAINT} which is a sharper version of our main result for curves in $E^2$. This is crucial for the applications and we   use it to prove Theorem~\ref{caso_E^N}. Section~\ref{DIMTHMMAIN}  is dedicated to the proof of Theorem  \ref{MAINT}.  Sections \ref{sez:criterio}--\ref{SecEx2} are devoted to describe the  families of examples and applications of our main method, proving in particular  Theorem~\ref{polyintro} and Theorem~\ref{curveEsp}.

\section{Notation and preliminaries}\label{notaz}
  In this section we introduce the  notations that we will use in the rest of the article. We define different heights and, among the main technical tools in the theory of height, we recall the Arithmetic B\'ezout Theorem and the Zhang inequality.  We also recall some standard facts on subgroups of $E^N$  and give  some basic  estimates for the degree of the kernel of morphisms on $E^N$.

\bigskip

\subsection{Heights and degrees}\label{heightsanddegrees}
In this article we deal only with varieties defined over the algebraic numbers. We will always identify a variety $V$ with the set of its algebraic points $V(\Qbar)$. Throughout the article $E$ will be an elliptic curve defined over the algebraic numbers and given by a fixed Weierstrass equation
\begin{equation}\label{EqWeierstrass}
E: y^2=x^3+Ax+B
\end{equation}
with $A$ and $B$  algebraic integers  (this assumption is not restrictive). If $E$ is defined over a number field $k$ we write in short $E/k$. As usual, we define the discriminant of $E$ as
\[
 \Delta=-16(4A^3+27B^2)\] and the $j$-invariant  \[j=\frac{-1728(4A)^3}{\Delta}.
\]
We also define
\begin{equation}\label{hWeierE}
h_{\mathcal{W}}(E)=h_W(1:A^{1/2}:B^{1/3})
\end{equation} to be the absolute logarithmic Weil height of the projective point $(1:A^{1/2}:B^{1/3})$. We recall that
 if $k$ is a number field, $\mathcal{M}_k$ is the set of places of $k$ and $P=(P_1:\ldots:P_n)\in \mathbb{P}_n(k)$ is a point in the projective space, then  the absolute logarithmic Weil height of $P$ is defined as
\[h_W(P)=\sum_{v\in \mathcal{M}_k} \frac{[k_v:\Q_v]}{[k:\Q]}\log \max_i \{\abs{P_i}_v\}.\]
We also consider a modified version of the Weil height, differing from it at the archimedean places
\begin{equation}\label{Defh2}
h_2(P)=\sum_{v\text{ finite}}\frac{[k_v:\Q_v]}{[k:\Q]}\log \max_i \{\abs{P_i}_v\} +\sum_{v\text{ infinite}}\frac{[k_v:\Q_v]}{[k:\Q]}\log \left(\sum_i \abs{P_i}_v^2\right)^{1/2}.\end{equation}
If $x$ is an algebraic number, we denote by $h_{\infty}(x)$ the contribution to the Weil height coming from the archimedean places, more precisely
  \[
 h_\infty(x)=\sum_{v\text{ infinite}}\frac{[k_v:\Q_v]}{[k:\Q]} \max \{\log \abs{x}_v, 0\}.\]

 To compute heights and degrees of  subvarieties of $E^N$, we consider them as embedded in $\P_{3^N-1}$ via the following composition  of maps
\begin{equation*}
  E^N \hookrightarrow \P_2^N \hookrightarrow \P_{3^N-1}
\end{equation*}
where the first map sends a point $(X_1,\ldots,X_N)$ to $((x_1,y_1),\dotsc,(x_N,y_N))$, the $(x_i,y_i)$ being the affine coordinates of $X_i$ in the Weierstrass form of $E$, and the second map is the Segre embedding.

\medskip

 For $V$ a subvariety of $E^N$ we consider the canonical height $h(V)$, as defined in \cite{patriceI}; when the variety $V$ reduces to a point $P$, then $h(V)=\hat h(P)$ is the N\'eron-Tate height of the point (see \cite{patriceI}, Proposition 9)  defined as $$\hat h(P)=\lim_{n\rightarrow \infty} \frac{h_W(2^n\cdot P)}{4^n}.$$

In general if $P=(P_1,\dotsc,P_N)\in E^N$, then we have
\begin{equation*}
   h(P)=\sum_{i=1}^N h(P_i)
\end{equation*}
 for $h$ equal to $h_{W},h_2$ and $\hat{h}$.

\medskip

For a subvariety $V\subseteq\P_m$  we denote by $h_2(V)$ the normalised height of $V$ {defined}  in terms of the Chow form of the ideal of $V$, as done  in  \cite{patrice}. This height extends the height $h_2$ defined for points by formula \eqref{Defh2}  (see \cite{Hab08}, page 6 and \cite{BGSGreen}, equation (3.1.6)).

 If $V$ is defined as an irreducible component of the zero-set in $\P_m$ of homogeneous polynomials $f_1,\ldots,f_r$, then by the result at page 347 and Proposition 4 of \cite{patrice} and standard estimates, one can prove that
\[h(V)\leq \sum_{i=1}^r h_W(f_i)\prod_{j\neq i} \deg(f_j)+c \deg(f_1)\cdots \deg(f_r)\]
where $h_W(f_i)$ is the Weil height of the vector of coefficients of $f_i$, considered as a projective point and $c$ is an explicit constant, which can be taken as $c=4m\log m$.

\medskip

The degree of  an irreducible variety $V\subseteq \P_m$ is the maximal cardinality of a finite intersection $V\cap L$, with $L$ a linear subspace of dimension equal to the codimension of $V$.

 The degree is often conveniently computed as an intersection product; we show here how to do it for a curve $\Ci\subseteq\P_2^N$.

Let $L$ be the class of a line in the Picard group of $\P_2$ and let $\pi_i:\P_2^N\to \P_2$ be the projection on the $i$-th component. Set $\ell_i=\pi_i^*(L)$.
The $\ell_i$'s have codimension 1 in $\P_2^N$ and they generate its Chow ring, which is isomorphic as a ring to $\Z[\ell_1,\dotsc,\ell_N]/(\ell_1^{3},\dotsc,\ell_N^{3})$.

The pullback through the Segre embedding of a hyperplane of $\P_{3^N-1}$ is given by $\ell_1+\dotsb+\ell_N$ as can be seen directly from the the equation of a coordinate hyperplane in $\P_{3^N-1}$.
The degree of $\Ci$ is therefore given by the intersection product $\Ci.(\ell_1+\dotsb+\ell_N)$ in the Chow ring of $\P_2^N$.

Assume now that $\Ci_i:=\pi_i(\Ci)$ is a curve for all $i$; by definition, $\deg\Ci_i = \deg(\Ci_i.L)$.

We see that
\begin{equation*}
 \pi_{i*}(\Ci.\ell_i)=\pi_{i*}(\Ci.\pi_i^*(L))=\pi_{i*}(\Ci).L=d_i\Ci_i.L
\end{equation*}
where $d_i$ is the degree of the map $\Ci\to\Ci_i$ given by the restriction of $\pi_i$ to $\Ci$, and the equality in the middle is given by the projection formula  (see \cite{Fulton}, Example 8.1.7). Taking the degrees we have
\[
\deg(\Ci.\ell_i)=\deg(\pi_{i*}(\Ci.\ell_i))=d_i\deg\Ci_i
\]
so that
\[
 \deg\Ci=\deg(\Ci.(\ell_1+\dotsb+\ell_N))=\deg(\Ci.\ell_1)+\dotsb+\deg(\Ci.\ell_N)=d_1\deg\Ci_1+\dotsb+d_N\deg\Ci_N.
\]

If in particular the curve $\Ci$ is contained in $E^N$, then all the $\Ci_i$'s are equal to $E$ and have degree 3.

 Notice that this formula remains true if for some of the $i$'s the restriction of $\pi_i$ to $\Ci$ is constant, provided that we take 0 as the degree of a constant map.

\medskip

We recall now two classical results on heights that will be important in the proof of our  theorems.
The first is an explicit version of the Arithmetic B\'ezout Theorem, as proved in \cite{patrice}, Th\'eor\`eme 3:
\begin{thm}[Arithmetic B\'ezout Theorem]\label{AriBez}
 Let $X$ and $Y$ be irreducible closed subvarieties of $\P_m$ defined over  the algebraic numbers. If $Z_1,\dotsc,Z_g$ are the irreducible components of $X\cap Y$, then
 \[
  \sum_{i=1}^g h_2(Z_i)\leq\deg(X)h_2(Y)+\deg(Y)h_2(X)+\csei(\dim X,\dim Y, m)\deg(X)\deg(Y)
 \]
where
\begin{equation*}
 \csei(d_1,d_2,m)=\left(\sum_{i=0}^{d_1}\sum_{j=0}^{d_2} \frac{1}{2(i+j+1)}\right)+\left(m-\frac{d_1+d_2}{2}\right)\log2.
\end{equation*}

\end{thm}

The second result is Zhang's inequality. In order to state it, we define the essential minimum $ \mu_2(X)$ of an irreducible algebraic subvariety $X\subseteq\P_m$ as

\[
  \mu_2(X)=\inf\{\theta\in\R\mid\{P\in X\mid  h_2(P)\leq\theta\}\text{ is Zariski dense in }X\}.
\]
The following is  a special case of \cite{Zhang95}, Theorem 5.2:
\begin{thm}[Zhang's inequality]\label{Zhang}
Let $X\subseteq\P_m$ be an irreducible algebraic subvariety. Then
\begin{equation}\label{zhangh2}
\mu_2(X)\leq\frac{h_2(X)}{\deg X}\leq(1+\dim X)\mu_2(X).
\end{equation}
\end{thm}
We also define a different essential minimum for subvarieties of $E^N$, relative to the height function $\hat h$:
\[
 \hat\mu(X)=\inf\{\theta\in\R\mid\{P\in X\mid \hat h(P)\leq\theta\}\text{ is Zariski dense in }X\}.
\]
Using the definitions and a simple limit argument, one sees that Zhang's inequality holds also with $\hat{\mu}$, namely
\begin{equation}\label{Zhangh^}\hat\mu(X)\leq\frac{h(X)}{\deg X}\leq(1+\dim X)\hat\mu(X).\end{equation}

\subsection{Algebraic Subgroups of $E^N$}\label{absub}
We recall that the uniformisation theorem  implies that $E(\C)$ is isomorphic, as complex Lie group, to $\C/\Lambda$ for a unique lattice  $\Lambda \subset \C$. The $N$-th power of this isomorphism gives the analytic uniformisation $\C ^N / {\Lambda}^N \stackrel{\sim}{\rightarrow} E^N(\C )$
of $E^N$ (see for instance  \cite{SilvermanArithmeticEllipticCurves}, Section VI, Theorem 5.1 and Corollary 5.1.1).
  Through the exponential map from the tangent space of $E^N$ at the origin to $E^N$, the Lie algebra of an abelian subvariety of $E^N$ is identified with a complex vector subspace $W \subset \C ^N$ for which $W \cap {\Lambda}^N$ is a lattice of full rank in $W$. The \emph{orthogonal complement} $B^\perp$ of an abelian subvariety $B \subset E^N$  is the abelian subvariety  with Lie algebra corresponding to  the orthogonal complement of the Lie algebra of $B$ with respect to the canonical Hermitian structure of $\C ^N$ (see for instance \cite{BG06} 8.2.27 and 8.9.8 for more details).

\section{Basic estimates for heights}\label{altezze}
This is a self-contained technical section in which we give several explicit estimates on heights, used later. The  readers who wish to skip these technical results may refer to the following table for the definition of the relevant  constants. The notation  was introduced in Section \ref{notaz}.

\subsection*{Summary of Constants}

  For ease of reference, we collect here  the definition of the constants $\ccinque,\dotsc,\cdieci$ that will intervene in our computations. Some of these quantities have a sharper expression when the curve $E$ is  defined over $\Q$
   and we deal with rational points.

\begin{table}[htb]
\newcommand{\mc}[3]{\multicolumn{#1}{#2}{#3}}

\begin{center}
\caption{}\label{table1}
\begin{tabular}[c]{|c|>{\centering\arraybackslash}m{5cm}|>{\centering\arraybackslash}m{7.5cm}|} 
\hline
&\vspace{0.02cm}&\\
   & For  $E/\overline{\Q}$ and $P\in E(\overline{\Q})$ & For $E/\Q$  and $P\in E(\Q)$\\
\hline
&\vspace{0.1cm}&\\
\centering
$\ccinque(E)$ & $\frac{h_W(\Delta)+h_\infty(j)}{4}+\frac{h_W(j)}{8}+$ \newline $+\frac{h_W(A)+h_W(B)}{2}+3.724$ & $\min\left(\frac{\log\abs{\Delta}+h_\infty(j)}{4}+\frac{h_W(j)}{8}+\frac{\log(\abs{A}+\abs{B}+3)}{2}+\right.$ $\left. \phantom{minn}+2.919, {\phantom{\frac{\frac{1}{1}}{1}}} 3h_{\mathcal W}(E)+4.709  \right)$\\
&\vspace{0.1cm}&\\
\hline
&\vspace{0.1cm}&\\
\centering
$\cquattro(E)$ & $\frac{h_W(\Delta)+h_\infty(j)}{4}+\frac{h_W(A)+h_W(B)}{2}+4.015$ & $\min\left(\frac{\log\abs{\Delta}+h_\infty(j)}{4}+\frac{\log(\abs{A}+\abs{B}+3)}{2}+ 3.21,\right.$ $\left. \phantom{minn} \frac{3h_{\mathcal W}(E)}{2}+2.427\right)$\\
&\vspace{0.1cm}&\\
\hline
&\vspace{0.1cm}&\\
$\csette(E)$ & {$\frac{h_W{(\Delta)}}{12}+\frac{h_\infty(j)}{12}+1.07$}&\\
&\vspace{0.1cm}&\\
\hline
&\vspace{0.1cm}&\\
$\cuno(E)$ & {$\frac{h_W(j)}{24}+\frac{h_W{(\Delta)}}{12}+\frac{h_\infty(j)}{12}+0.973$}&\\
&\vspace{0.1cm}&\\
\hline
&\vspace{0.1cm}&\\
$\cnove(E)$ & $\ccinque(E)$ & $3h_{\mathcal{W}}(E)+6\log2$\\
&\vspace{0.1cm}&\\
\hline
&\vspace{0.1cm}&\\
$\cotto(E)$ & $\frac{h_W(A)+h_W(B)+\log5}{2}$ & $\frac{\log( 3+\abs{A}+\abs{B})}{2}$\\
&\vspace{0.1cm}&\\
\hline
&\vspace{0.1cm}&\\
$\cdieci(E)$ & $\frac{h_W(A)+h_W(B)+\log3}{2}$ & $\frac{\log(1+\abs{A}+\abs{B})}{2}$\\
&\vspace{0.1cm}&\\
\hline
 \end{tabular}
\end{center}
\end{table}

 All the above constants are computed below. More precisely, the constants $\ccinque(E)$ and $\cquattro(E)$, first appearing in Theorem \ref{caso_E^N}, are computed in Proposition \ref{prop.confronto.h2h^}, by combining bounds of Silverman and Zimmer. The constants $\csette(E)$ and $\cuno(E)$ come from formula \eqref{BoundSilvermanAltezze} proved in \cite{SilvermanDifferenceHeights} Theorem 1.1. Moreover $\cnove(E)$  is given in Zimmer's bound  \cite{ZimmerAltezze},  p.~40 recalled in  \eqref{stima_zimmer}. Finally $\cotto(E)$ and $\cdieci(E)$ are proved in Lemma \ref{lem.confronto.h2hx}.

 \vspace{0.8cm}

We now give the details for determining these constants.

If $P$ is a point in $\mathbb{P}_m$, from the definition of $h_W$ and $h_2$, we have \begin{equation}\label{hWeil}h_W(P)\leq h_{2}(P)\leq h_W(P)+\log(m+1)/2.\end{equation}

If $P\in E$, then, from \cite{SilvermanDifferenceHeights}, Theorem 1.1, we have
\begin{equation}\label{BoundSilvermanAltezze}
-\cuno(E) \leq \frac{\hat h(P)}{3}-\frac{h_W(x(P))}{2}\leq \csette(E)
\end{equation}
where \[\csette(E)=\frac{h_W(\Delta)}{12}+\frac{h_\infty(j)}{12}+1.07\]
and \[\cuno(E)=\frac{h_W(j)}{24}+\frac{h_W(\Delta)}{12}+\frac{h_\infty(j)}{12}+0.973\]
(notice that the N\'eron-Tate height used by Silverman in \cite{SilvermanDifferenceHeights} is one third of our $\hat h$, as defined in \cite{patriceI}).

If $E/\Q$ and $P\in E(\Q)$, Zimmer \cite{ZimmerAltezze},  p.~40, proved that:
 \begin{equation}\label{stima_zimmer}
-\frac{3h_{\mathcal W}(E)}{2}-\frac{7}{2}\log 2\leq  h_{W}(P)-\hat h(P)\leq 3h_{\mathcal W}(E)+6\log 2.
\end{equation}

 We remark that Silverman's bound is better than Zimmer's one for elliptic curves with big coefficients. Nevertheless we included here Zimmer's estimates because they are sharper in  some of our examples.

 \medskip

In the following lemma we compare $h_2$ and $h_W$ for points in $E$.
\begin{lem}\label{lem.confronto.h2hx}
For every point $P\in E$ we have
 \begin{align}\notag\abs{h_2(P)-\frac{3}{2}h_W(x(P))}\leq \cotto(E),\\ \notag \abs{h_2(P)-h_W(y(P))}\leq \cotto(E),\\
 \notag \abs{h_W(y(P))-\frac{3}{2}h_W(x(P))}\leq \cdieci(E)
 \end{align}
where
$$\cotto(E)= \frac{h_W(A)+h_W(B)+\log 5}{2}$$
and $$\cdieci(E)=\frac{h_W(A)+h_W(B)+\log 3}{2}.$$
  If moreover $E/\Q$  we may take the sharper values
$$\cotto(E)=\frac{\log( \abs{A}+\abs{B}+3)}{2}$$
and $$\cdieci(E)=\frac{\log(\abs{A}+\abs{B}+1)}{2}.$$
\end{lem}
\begin{proof}
 We write both $h_W$ and $h_2$ in terms of local contributions and bound each of them.
 Let $P=(x,y)\in E$ and let $k$ be a number field of definition for $P$ and $E$. Let us first compare $h_2(P)$ and $h_W(x(P))$.

For every place $v$ of $k$, we set $\lambda_v=[k_v:\Q_v]/[k:\Q]$.

 By the definitions of $h_W$ and $h_2$, if $v$ is a non-archimedean place, then the contribution to the difference $h_2(P)-\frac{3}{2}h_W(x(P))$ coming from $v$ is
 \[\lambda_v\left(\log \max ( 1, \abs{x}_v,\abs{y}_v)-\frac{3}{2}\log \max ( 1, \abs{x}_v)\right).\]
 We see that if $\abs{x}_v\leq 1$ then $\abs{y}_v\leq 1$ as well, because $A$ and $B$ are algebraic integers, and this contribution is 0. If instead $\abs{x}_v> 1$, then $\abs{y}_v^2=\abs{x^3+Ax+B}_v=\abs{x}_v^3$ thanks to the ultrametric inequality, and the contribution is again 0.

 If $v$ is an archimedean place, then the contribution coming from $v$ is
 \begin{align*}
  &\lambda_v\left(\frac{1}{2}\log (1+ \abs{x}_v^2+\abs{y}_v^2)-\frac{3}{2}\log \max ( 1, \abs{x}_v)\right)=\\
  =&\frac{\lambda_v}{2}\left(\log (1+ \abs{x}_v^2+\abs{x^3+Ax+B}_v)-3\log \max ( 1, \abs{x}_v)\right).
 \end{align*}
If $\abs{x}_v\leq 1$ this quantity is at most $\frac{\lambda_v}{2}\log (3+\abs{A}_v+\abs{B}_v)$.
If $\abs{x}_v> 1$ we write
\[\log (1+ \abs{x}_v^2+\abs{x^3+Ax+B}_v)-3\log \abs{x}_v=\log\left(\frac{1}{\abs{x}_v^3}+\frac{1}{\abs{x}_v}+\abs{1+\frac{A}{x^2}+\frac{B}{x^3}}_v\right),\]
 which is again at most  $\frac{\lambda_v}{2}\log (3+\abs{A}_v+\abs{B}_v)$.
 If $E$  is defined over $\Q$, then  the sum of all $\lambda_v$, for $v$ ranging in the archimedean places, is 1 and we get the bound in the statement. If this is not the case, then we check that
 \[\log(3+a+b)\leq\log 5 +\max(0,\log a)+\max(0,\log b) \qquad \forall a,b>0\]
 so that the difference is bounded by
 \[\sum_{v \text{ archimedean}}\lambda_v\frac{\max(0,\log\abs{A}_v)+\max(0,\log\abs{B}_v)+\log 5}{2}=\frac{h_W(A)+h_W(B)+\log 5}{2}.\]

 Let us now compare $h_2(P)$ and $h_W(y(P))$. Just as in the case discussed above, the non-archimedean absolute values give no contribution. Let $v$ be an archimedean absolute value. The quantity to bound is
  \begin{align*}
  &\lambda_v\left(\frac{1}{2}\log (1+ \abs{x}_v^2+\abs{y}_v^2)-\log \max ( 1, \abs{y}_v)\right).
 \end{align*}
 We consider two cases:

 If $\abs{x}_v^2\leq 1+\abs{A}_v+\abs{B}_v$ then one easily checks that
 \[\frac{1}{2}\log (1+ \abs{x}_v^2+\abs{y}_v^2)-\log \max ( 1, \abs{y}_v)\leq\frac{1}{2} \log(3+\abs{A}_v+\abs{B}_v)\]
 for all values of $\abs{y}_v$.

 If $\abs{x}_v^2 > 1+\abs{A}_v+\abs{B}_v$ then \[\abs{y}^2_v\geq\abs{x}_v^2\abs{x}_v-\abs{Ax}_v-\abs{B}_v>\abs{x}_v+\abs{B}_v\abs{x}_v-\abs{B}_v>\abs{x}_v>1\]
 and therefore the quantity to bound is
   \begin{align*}
  \frac{\lambda_v}{2}\log \left(1+ \frac{\abs{x}_v^2+1}{\abs{x^3+Ax+B}_v}\right).
 \end{align*}
 To see that
 \[\abs{x}_v^2+1\leq(2+\abs{A}_v+\abs{B}_v)\cdot\abs{x^3+Ax+B}_v\]
 we write
 \begin{multline*}
  (2+\abs{A}_v+\abs{B}_v)\cdot\abs{x^3+Ax+B}_v\geq\\
  \geq(2+\abs{A}_v+\abs{B}_v)\abs{x}_v^3-(2+\abs{A}_v+\abs{B}_v)(\abs{Ax}_v+\abs{B}_v)>\\
  >\abs{x}_v^3+(1+\abs{A}_v+\abs{B}_v)^2\abs{x}_v-(2+\abs{A}_v+\abs{B}_v)(\abs{Ax}_v+\abs{B}_v)  \geq \abs{x}_v^2+1.
 \end{multline*}
The bound in the statement now follows as in the first case.
The bound between $h_W(x(P))$ and $h_W(y(P))$ is proved analogously.
\end{proof}

The following proposition combines in a single statement the bounds by Silverman and Zimmer that we recalled before and Lemma \ref{lem.confronto.h2hx}. It gives a bound between $\hat{h}$ and $h_2$  for a point in $E^N$.   This estimate is used in the proof of our main theorem.

\begin{propo}\label{prop.confronto.h2h^}
Let  $P\in E^N$. Then
 \begin{equation*}
 -N\cquattro(E)\leq h_2(P)-\hat h (P)\leq N\ccinque(E),
\end{equation*}
where
$$ \ccinque(E)=\frac{h_W(\Delta)+h_\infty(j)}{4}+\frac{h_W(j)}{8}+\frac{h_W(A)+h_W(B)}{2}+3.724,$$
$$ \cquattro(E)=\frac{h_W(\Delta)+h_\infty(j)}{4}+\frac{h_W(A)+h_W(B)}{2}+4.015.$$
Moreover, if $E/\Q$ and $P\in E(\Q)$ one can take
$$\ccinque(E)=\min\left(\frac{\log\abs{\Delta}+h_\infty(j)}{4}+\frac{h_W(j)}{8}+\frac{\log(\abs{A}+\abs{B}+3)}{2}+2.919,3h_{\mathcal W}(E)+4.709\right),$$
$$ \cquattro(E)=
                \min\left(\frac{\log\abs{\Delta}+h_\infty(j)}{4}+\frac{\log(\abs{A}+\abs{B}+3)}{2}+ 3.21, \frac{3h_{\mathcal W}(E)}{2}+2.427\right).$$
\end{propo}
\begin{proof}
The general bounds are obtained by \eqref{BoundSilvermanAltezze} combined with Lemma \ref{lem.confronto.h2hx}.
When $E$  is defined over $\Q$ and the point $P\in E(\Q)$, they can be sharpened by taking the minimum between the bounds obtained combining \eqref{BoundSilvermanAltezze}  with Lemma \ref{lem.confronto.h2hx} and the ones obtained combining \eqref{stima_zimmer}  with  \eqref{hWeil}.
\end{proof}

\bigskip

Using Proposition \ref{prop.confronto.h2h^} we immediately deduce the following relation  between the two essential minima $\mu_2(X)$ and $\hat{\mu}(X)$ introduced in Section \ref{notaz}, for any  irreducible subvariety $X$ of $E^N$. We have
\begin{equation}\label{mu2mu^}
-N\cquattro(E)\leq \mu_2(X)-\hat\mu(X)\leq  N\ccinque(E)
\end{equation}
where the constants are defined in Proposition \ref{prop.confronto.h2h^}.

Finally, using \eqref{mu2mu^}, \eqref{zhangh2} and \eqref{Zhangh^} we  get: \begin{equation}\label{confrontohh2}
 \frac{h_2(X)}{1+\dim X}-N\ccinque(E)\deg X\leq h(X)\leq (1+\dim X)\left(h_2(X)+N\cquattro(E)\deg X\right).
\end{equation}

\section{Main results and consequences}\label{nuova}\label{deduz}
 In this section we prove a sharper version of Theorem \ref{caso_E^N}. The proof  relies on a geometrical induction on the dimension $N$ of the ambient variety. We split the statement and the proof in two parts: the base of the induction  given by $N=2$ is  Theorem \ref{MAINT}, and we postpone its proof to Section~\ref{DIMTHMMAIN}; the inductive step given for  $N\geq 3$ is  Theorem \ref{IS}.
Finally we give some  more general formulations of our main theorem and  additional remarks.

It is evident that our Theorem \ref{caso_E^N} in the Introduction is a direct consequences of Theorems \ref{MAINT} and  \ref{IS}, where the bounds in the Theorem \ref{caso_E^N} are less sharp. This sharper version and the finer constants for points overs $\Q$ are used in the applications to keep the bounds for the height of the rational points on a curve as small as possible.

In our context, we  characterise arithmetically  points by their rank (see Definition \ref{rank}), while geometrically we characterise a  curve by its transversality property.

\begin{D}\label{deftrans}
A curve $\mathcal{C}$ in an abelian variety $A$ is  transverse (resp. weak-transverse) if it is irreducible and it is not contained in any translate (resp. in any torsion variety).

Here by translate (resp. torsion variety) we mean a finite union of translates of proper algebraic subgroups of $A$ by points (resp.  by torsion points).
\end{D}

We remark that curves of genus $1$ are translates of an elliptic curve and  that, in an abelian variety $A$ of dimension 2, a curve has genus at least $2$ if and only if it is transverse. Thus, for $\Ci$ in $E^2$ assuming transversality is equivalent to the assumption that the genus is at least $2$.  Then it is equivalent to state the following theorem for transverse curves.
\begin{thm}[Base of the reduction]\label{MAINT}
Let $E$ be an elliptic curve without CM. Let $\Ci$ be  an irreducible curve in $E^2$ of genus $\ge 2$.
Then every point $P$ on $\Ci$ of rank  $\le1$  has height bounded as:
   \begin{equation*}
  \hat h(P)\leq \Cuno\cdot h_2(\Ci)\deg\Ci +\Cdue(E)(\deg\Ci)^2 +\Ctre(E)
\end{equation*}
  where
\begin{align*}
    \Cuno&=72.251\\
    \Cdue(E)&=\Cuno\left(6.019+4\ccinque(E)\right)\\
    \Ctre(E)&=4\cquattro(E),
 \end{align*}
  and  the constants $\ccinque(E)$ and $\cquattro(E)$ are defined in Table \ref{table1}.
\end{thm}

The proof of this theorem is the content of the following Section \ref{DIMTHMMAIN}.\\

 We now show how to use Theorem~\ref{MAINT} to prove the following sharper version of our main Theorem \ref{caso_E^N} for $N\ge3$. The central idea is to argue by induction and project $\Ci$ from $E^N$ to $E^n$ for $n<N$ in such a way that the projection is transverse and its height and degree are well controlled. In order to obtain better bounds, we study different cases according to the geometric conditions satisfied by $\Ci$.
  \begin{thm}[Reduction Step]
 \label{IS}
  Let $E$ be an elliptic curve without CM. Let $N\ge3$ be an integer.
 If $\Ci$ is  an irreducible curve of genus at least $2$ embedded in $E^N$, then every point  $P$ of  rank at most one in $\Ci$ has N\'eron-Tate height bounded as
 \begin{align*}
  \hat h(P)\leq& 2\cdot3^{N-2} N! \Cuno h_2(\Ci)(\deg\Ci)^2 +\frac{3^{N-2} N!}{2}\Cdue(E)(\deg\Ci)^3 +3^{N-2}(N-2)!h_2(\Ci)+\\&+\deg\Ci (3^{N-2} (N-2)!)\left(N(N-1)\left(\frac{\Ctre(E)}{2}+\ccinque(E)\right)+C_0(N)\right)+N\cquattro(E).
 \end{align*}
 If $\Ci$ is weak-transverse we get
      \begin{equation*}
  \hat h(P)\leq 4 (N-1) \Cuno h_2(\Ci)\deg\Ci +(N-1) \Cdue(E)(\deg\Ci)^2 +N^2\Ctre(E).
 \end{equation*}
  If furthermore $\Ci$ is transverse, then    \begin{equation*}
  \hat h(P)\leq N \Cuno h_2(\Ci)\deg\Ci + \frac{N}{2} \Cdue(E)(\deg\Ci)^2 +\frac{N}{2}\Ctre(E).
 \end{equation*}
Here  \begin{align*}
\csei(N)&=(3^N-3/2) \log 2 +\sum_{i=1}^{N-1}\frac{1}{i}-\frac{1}{2N}\\
    \Cuno&=72.251\\
    \Cdue(E)&=\Cuno\left(6.019+4\ccinque(E)\right)\\
    \Ctre(E)&=4\cquattro(E),
 \end{align*}
 and the constants $\ccinque(E)$ and $\cquattro(E)$ are defined in Table \ref{table1}.
 \end{thm}

  \begin{proof}
If $P$ has rank 0 then it is a torsion point and its height is trivial. So we assume that $P$ has rank  one.

   We first suppose that $\Ci$ is also transverse in $E^N$.
 Let $\pi:E^N\to E^2$ be the projection on any two coordinates. Since $\Ci$ is transverse in $E^N$, then $\pi(\Ci)$ is a  transverse curve in $E^2$.

 By Lemma 2.1 in \cite{MW90} we have that $\deg\pi(\Ci)\leq \deg\Ci$.
 Clearly $h_2(\pi(P))\leq h_2(P)$ for every point $P$ in $E^N$, therefore $\mu_2(\pi(\Ci))\leq\mu_2(\Ci)$.  By Theorem~\ref{Zhang}, we have  that \[h_2(\pi(\Ci))\leq 2\mu_2(\pi(\Ci))\deg\pi(\Ci)\leq 2 \mu_2(\Ci)\deg\Ci\leq 2h_2(\Ci).\]

Let now $P=(P_1,\dotsc,P_N)\in\Ci$ be a point of rank one. Up to a reordering of the factors of $E^N$ we may assume that $\hat h(P_1)\geq \hat h(P_2)\geq\dotsb\geq\hat h(P_N)$ and let $\pi$ be the projection on the first two coordinates.
  Then \begin{equation}\label{boundP1P2}\hat h(P)\leq \hat h(P_1)+(N-1)\hat h(P_2)\leq \frac{N}{2}\hat h(\pi(P)).\end{equation}
  We apply Theorem~\ref{MAINT} to bound the height of  $\pi(P)$ on $\pi(\Ci)$ in $E^2$, obtaining
  \begin{align*}
  \hat h(\pi(P))&\leq \Cuno\cdot h_2(\pi(\Ci))\deg\pi(\Ci) +\Cdue(E)(\deg\pi(\Ci))^2 +\Ctre(E)\\
                &\leq 2 \Cuno\cdot h_2(\Ci)\deg\Ci +\Cdue(E)(\deg\Ci)^2 +\Ctre(E).
  \end{align*}

 Substituting this estimate in formula \eqref{boundP1P2} we get the wished bound for $\Ci$  transverse.\\

 Suppose now  that $\Ci$ is weak-transverse, but it is not transverse.  If the set of points of $\Ci$ of rank one is empty   nothing has to be proven.   We show that if  it   is not empty, then we can  reduce to the case of  a transverse curve in $E^{N-1}$.

 Since $\Ci$ is not transverse,   but weak-transverse, it is contained in a proper non-torsion translate of minimal dimension $H+Q$, where $H$ is a proper abelian subvariety of $E^N$  and  $Q$ is a point in   the orthogonal complement $H^{\perp}$ of $H$,  defined in Section \ref{absub}.

 We now prove that $\dim H^{\perp}=1$.
 Let $P$ be a point of $\Ci$ of rank one. Since  $Q$ is the component  of $P$ in $H^{\perp}$, we deduce  that $Q$ has rank at most one. But $Q$ cannot be torsion, so it has rank one and    $\dim H^\perp=1$.

Up to a reordering of the coordinates of  $P=(P_1,\dotsc,P_N)$, we can assume that    $\hat h(P_1)\geq \hat h(P_i)$ for all $i=1,\dotsc,N$.
We denote by $\pi_i:E^N\to E^{N-1}$ the natural projection which omits the $i$-th coordinate.

Assume first that there exists an index $i\neq 1$ such that the restriction of $\pi_i$ to $H$ is surjective. In this case $\pi_i(\Ci)$ is a transverse curve in $E^{N-1}$.
We easily see that $\mu_2(\pi_i(\Ci))\leq \mu_2(\Ci)$; by Lemma~2.1 of \cite{MW90} $\deg\pi_i(\Ci)\leq\deg\Ci$; by Zhang's inequality  $h_2(\pi_i(\Ci))\leq 2h_2(\Ci)$.

So if $N=3$ we apply Theorem \ref{MAINT} and if $N>3$ we apply the first part of the proof to $\pi_i(\Ci)$ transverse in $E^{N-1}$ obtaining
\begin{align*}
   \hat h(\pi_i(P))&\leq (N-1)\Cuno\cdot h_2(\pi_i(\Ci))\deg\pi_i(\Ci) +\frac{N-1}{2}\Cdue(E)(\deg\pi_i(\Ci))^2 +\frac{N-1}{2}\Ctre(E)\\
               &\leq 2(N-1)\Cuno\cdot h_2(\Ci)\deg\Ci +\frac{N-1}{2}\Cdue(E)(\deg\Ci)^2 +\frac{N-1}{2}\Ctre(E).
 \end{align*}
 Moreover, the height of $P$ is easily bounded as $\hat h(P)\leq 2\hat h(\pi_i(P))$, because the first coordinate has maximal height for $P$ and it is in the projection as $i\not=1$. This gives the desired bound for $\Ci$ weak-transverse.

 We are left with the case where  the restriction of $\pi_i$ to $H$ is not surjective for all $i\neq 1$. Then $H\supseteq\ker\pi_i$ for all $i\neq 1$ and  by counting dimensions $H=\{O\}\times E^{N-1}$. Therefore $Q$ is, up to a torsion point,  the first component $P_1$ of the point $P$ and
 \begin{equation}\label{eq.15}
    \hat h(P)\leq N\hat h(P_1)= N\hat h(Q).
 \end{equation}
Using  \cite{preprintPhilippon} we obtain $\hat h(Q)=\hat\mu(\Ci)-\hat\mu(\Ci-Q)\leq\hat\mu(\Ci)\leq\frac{h(\Ci)}{\deg\Ci}$.  Substituting this in \eqref{eq.15} and using \eqref{confrontohh2} we have
 \[
    \hat h(P)\leq N\frac{h(\Ci)}{\deg\Ci}\leq 2N\left(\frac{h_2(\Ci)}{\deg \Ci}+N\cquattro(E)\right),
 \]
 where $\cquattro(E)$ is defined in Table \ref{table1}. This concludes the weak-transverse case as this  bound is smaller then the one in the statement.

\medskip

We finally treat the case of $\Ci$ of genus at least $2$, but not weak-transverse. Let $H+Q$ be the translate of smallest dimension containing $\Ci$ with $Q\in H^{\perp}$, where this time there are no conditions on the rank of $Q$. Then $\Ci-Q$ is transverse in $H$ and  the dimension of $H$ is at least $2$ otherwise $\Ci=H+Q$ would have genus $1$. Consider the natural projections $\pi:E^N\to  E^{\dim H}$ that omit some $d=N-\dim H$ coordinates. For a question of dimensions, at least one projection $\pi$ is surjective when restricted to $H$. Thus the image  $\pi(\Ci)$ is transverse in $E^{\dim H}$.   Moreover, like in the previous cases, we have $\deg \pi(\Ci)\leq \deg\Ci$ and $h_2(\pi(\Ci))\leq 2h_2(\Ci)$. We can then apply   the first part of the proof to $\pi(\Ci)$ transverse in $E^{\dim H}$ to get
\begin{equation}\label{eq1}\hat h(\pi(P))\leq 2 (N-d) \Cuno h_2(\Ci)\deg\Ci + \frac{N-d}{2} \Cdue(E)(\deg\Ci)^2 +\frac{N-d}{2}\Ctre(E).\end{equation}
 To bound $h_2(P)$ we first remark that $P$ is a component of  $\Ci \cap (\ker \pi+\pi(P)),$ otherwise $\Ci-Q \subseteq \ker \pi+\pi(P)\cap H \subsetneq H$ would not be transverse in $H$. We then use the Arithmetic B\'ezout  Theorem for  $\Ci \cap (\ker \pi+\pi(P)),$  where we bound   $h_2(\ker \pi+\pi(P))$ using  Zhang's Inequality,   equation \eqref{mu2mu^}  and that  $\hat \mu(\ker \pi+\pi(P))=\hat h(\pi(P))$ by \cite{preprintPhilippon}. All of this gives
\begin{align}\label{bh2Ppreciso}
h_2(P)\leq& (1+d)(\deg\ker\pi)\hat h(\pi(P))\deg\Ci+(\deg\ker\pi) h_2(\Ci)+\\
\notag&+\left((1+d)N\ccinque(E)+\csei(1,d,3^N-1)\right)(\deg\ker\pi) \deg\Ci,
\end{align}
where
 \begin{align*}
  \csei(1,d,3^N-1)&=\sum_{i=1}^{d+2}\frac{1}{i}-\frac{d+3}{2(d+2)}+\left(3^N-\frac{d+3}{2}\right)\log 2
 \end{align*}
where is defined in the Arithmetic B\'ezout Theorem \ref{AriBez}.

 Clearly $d\leq N-2$ and $\deg\ker\pi=3^{d}d!\leq 3^{N-2}(N-2)!$, so
 setting \[\csei(N)=(3^N-3/2) \log 2 +\sum_{i=1}^{N-1}\frac{1}{i}-\frac{1}{2N}\] we have
\[
  \csei(1,d,3^N-1)\leq \csei(N)
\]
and
\begin{align}\label{bh2P}
h_2(P)\leq&  3^{N-2}(N-1)! \hat h(\pi(P))\deg\Ci+3^{N-2} (N-2)! h_2(\Ci)+\\
\notag&+3^{N-2} (N-2)!\left(N(N-1)\ccinque(E)+C_0(N)\right)\deg\Ci.
\end{align}

 Finally, substituting \eqref{eq1} into \eqref{bh2P} and using Proposition \ref{prop.confronto.h2h^} to compare $\hat h(P)$ and $h_2(P)$, we get the  bound in the statement.
 \end{proof}

Clearly, if $E$ and $\Ci$ are defined over $k$ and $E(k)$ has rank one then the set $\Ci(k)$ of $k$-rational points of $\Ci$ is a subset of  the set of points on $\Ci$ of rank one, thus of height bounded as above. We now show how a similar  strategy applies to curves transverse in an abelian variety   with a factor $E^2$.  The bounds are explicit when an embedding of the abelian variety in some projective space is given, even though this happens rarely for abelian varieties of higher dimension.
 \begin{propo}\label{prop.transv.A}
    Let $E$ be an elliptic curve   and $A$ an abelian variety, both defined over a number field $k$;  let $E$ be embedded in $\mathbb{P}_2$ through equation \eqref{EqWeierstrass} and let us fix an embedding of $A$  in some projective space.
  \begin{itemize}
\item[(a)]
  Assume that $E$ is without CM.
    Let $\Ci$ be a  curve transverse in $E^2\times A$.
    Then  every point  $P$ in $\Ci$ of rank at most one has:
\begin{align*}
   h_2(P)&\leq h_2(A)(1+\dim A)\deg\Ci+\deg A \left(h_2(\Ci)+\csei \deg \Ci\right)+(1+\dim A)\deg A\\
&\left(\left(\Ctre(E) +2\ccinque(E)\right)\deg\Ci+2\Cuno(E) h_2(\Ci)(\deg\Ci)^2+\Cdue(E)(\deg\Ci)^3\right).
 \end{align*}
  \item[(b)]
  Assume that  $E(k)$ has rank zero. Let $\Ci$ be a  curve over $k$ weak-transverse in $E\times A$. Then for every point  $P\in \Ci(k)$  we have:
\begin{align*}
   h_2(P)&\leq  (1+\dim A)\left(2\ccinque(E) \deg A +h_2(A)\right)\deg\Ci +\deg A h_2(\Ci)+\csei \deg A \deg \Ci.
 \end{align*}
  \item[(c)]
  Assume that $E$ is without CM and that $E(k)$ has rank one.
    Let $\Ci$ be a  curve over $k$ transverse in $E^2\times A$.
    Then for every point  $P\in \Ci(k)$ we have:
\begin{align*}
   h_2(P)&\leq h_2(A)(1+\dim A)\deg\Ci+\deg A \left(h_2(\Ci)+\csei \deg \Ci\right)+(1+\dim A)\deg A\\
&\left(\left(\Ctre(E) +2\ccinque(E)\right)\deg\Ci+2\Cuno(E) h_2(\Ci)(\deg\Ci)^2+\Cdue(E)(\deg\Ci)^3\right).
 \end{align*}
 \end{itemize}
  Here the constants $\Cuno, \Cdue(E), \Ctre(E)$ are defined in Theorem~\ref{caso_E^N},  $\csei$  in Theorem \ref{AriBez}  and $ \ccinque(E)$ in Table \ref{table1}.
 \end{propo}
\begin{proof}
 Part (c) is an immediate corollary   of part (a).

 To prove parts (a) and (b), we use Theorem \ref{MAINT} and  the same strategy as in the proof of Theorem~\ref{IS} above.

 Let $P$ be a point in $\Ci$ of rank one in case (a), respectively a $k$-rational point in case (b), and let  $\pi: E^2\times A \to E^2$ be the natural projection on $E^2$ for the case (a) and let  $\pi: E\times A \to E$ be the natural projection on  $E$ for the case (b).

   The point $P$ is a component of $\Ci$ intersected with $A'=\{\pi(P)\}\times A$, in case (a) because the curve $\Ci$ is transverse, in case (b) because $\Ci$ is weak-transverse and $\pi(P)$ is a torsion point. By the Arithmetic B\'ezout Theorem we deduce
 \begin{equation}\label{A3}
  h_2(P)\leq h_2(A')\deg\Ci +h_2(\Ci)\deg A' +\csei \deg \Ci\deg A'
 \end{equation}
 where the constant $\csei$ is explicitly given in Theorem \ref{AriBez}.

Clearly $\deg A'= \deg A$, so we are left to bound $h_2(A')$.

Using Zhang's Inequality
we get
  \begin{align}\label{formu}
   h_2(A')&\leq (1+\dim A)\deg A \mu_2(A')=  (1+\dim A)\deg A \left(h_2(\pi(P))+\mu_2(A)\right).
  \end{align}
 Moreover $\mu_2(A)\leq h_2(A)/\deg A$ and $h_2(\pi(P))\leq \hat h(\pi(P))+2\ccinque(E)$  by Proposition \ref{prop.confronto.h2h^}. Thus
  \begin{align}\label{A2}
   h_2(A')&\leq
     (1+\dim A)\left(\deg A \left(\hat h(\pi(P))+2\ccinque(E)\right)+h_2(A)\right).
  \end{align}

  In  case (b), $\pi(P)$ is a torsion point, so  $\hat h(\pi(P))=0$ and we directly deduce the bound.

  To bound $\hat h(\pi(P))$ in  case (a), we apply Theorem \ref{MAINT} to the curve $\pi(\Ci)$ transverse in $E^2$ and we use that $\deg\pi(\Ci)\leq \deg\Ci$ by \cite{MW90} Lemma 2.1 and  $\mu_2(\pi(\Ci))\leq\mu_2(\Ci)$ by the definition of essential minimum  and thus $h_2(\pi(\Ci))\le 2 h_2(\Ci)$ by Zhang's inequality. We obtain
 \begin{align}\label{A1}
  \hat h(\pi(P))\leq 2 \Cuno\cdot h_2(\Ci)\deg\Ci +\Cdue(E)(\deg\Ci)^2 +\Ctre(E).
  \end{align}

 Combining \eqref{A1},\eqref{A2} and \eqref{A3} we get the bound in part (a).

\end{proof}

  \begin{remark}\label{rem.rank0}
 Using the universal property of the Jacobian one can extend the above argument to any curve such that the Jacobian has a factor $E$ of rank zero or $E^2$  with $E$ of rank one.

 In addition in Proposition \ref{prop.transv.A} case (b), with $k=\Q$, the number of rational points of $\Ci$ is easily bounded using Mazur's theorem (\cite{Mazur}, Theorem 8) and B\'ezout's theorem, giving $$\sharp \Ci(\Q) \le  16 \deg A \deg \Ci.$$
 Similarly, a bound for the number of $k$-rational points can be given using B\'ezout Theorem and the bound of Parent (\cite{Parent99}) for the size of the torsion group in terms of the degree of $k$.

 \end{remark}

\bigskip

As a final remark in this section we notice that the bounds given in Theorem \ref{MAINT} use, among others, the estimates of Proposition \ref{prop.confronto.h2h^}. We give here a more intrinsic formulation of our result, where the dependence on the  height bounds of Proposition \ref{prop.confronto.h2h^} is explicitly given.
\begin{teo}\label{MAINTprimo}
 Let $E$ be an elliptic curve without CM. Let $\Ci$ be a transverse curve in $E^2$.
Let $\duno(E),\ddue(E)>0$ be two constants such that
  \begin{equation}\label{eq:1}
 -\duno(E)\leq h_2(Q)-\hat h (Q)\leq \ddue(E) \quad \forall Q\in E(\overline{\Q}).
\end{equation}
Then for every point  $P$ in $\Ci$ of rank at most one, we have:
\begin{equation*}
  \hat h(P)\leq \Duno\cdot h_2(\Ci)\deg\Ci +\Ddue(E)(\deg\Ci)^2 +\Dtre(E)
\end{equation*}
 where
 \begin{align*}
    \Duno&=72.251\\
    \Ddue(E)&=\Duno\left(6.019+4\duno(E)\right)\\
    \Dtre(E)&=4\ddue(E).
 \end{align*}
\end{teo}
This formulation might help for potential future applications; indeed in different elliptic curves one can prove different versions of the bounds in \eqref{eq:1} and possibly improve, in special cases, the bounds in our main theorem.

\section{The proof of the main Theorem for $N=2$}\label{DIMTHMMAIN}
In this section we first prove the new key estimate at the base of the bound in  Theorem \ref{MAINT} and then we show how to conclude its proof.

\subsection{Bounds for the degree and the height of a translate}\label{bounddegalt}
Here we  prove some general bounds for the degree and the height of a proper translate $H+P$ in $E^2$ in terms of $\hat h(P)$ and of the coefficients of the equation defining the algebraic subgroup $H$.

\begin{propo}\label{boundsotto}Let $P=(P_1,P_2)$ be a point in $E^2$, where $E$ is without CM. Let $H$ be a component of the  algebraic subgroup in $E^2$ defined by the equation $\alpha X_1+\beta X_2=O$, with $u=(\alpha,\beta)\in \mathbb Z^2\setminus\{(0,0)\}$. Then
$$\deg (H+P)\leq 3||{u}||^2$$
where $||{u}||$ denotes the euclidean norm of $u$,
$$h(H+P)\leq 6\hat h(u(P)),$$
 and
$$h_2(H+P)\leq 6\hat h(u(P))+12||u||^2\ccinque(E)$$
where $u(P)=\alpha P_1+\beta P_2$  and $\ccinque(E)$ is defined in Table \ref{table1}.
\end{propo}

\begin{proof}
{\bf A bound for the degree of $H+P$.}

We compute the degree of $H+P$ as explained in Section \ref{heightsanddegrees}.
 The map ${\pi_2}:H+P\to E$ has degree $\alpha^2$; then
\[\deg((H+P).\ell_2)=\alpha^2\deg E= 3\alpha^2.\]
The same holds exchanging $\alpha$ with $\beta$ and $\ell_2$ with $\ell_1$.
Therefore computing the degree as intersection product we get
\begin{equation}\label{gradoH+P}
  \deg (H+P)= 3(\alpha^2+\beta^2)=3||{u}||^2.
\end{equation}

{\bf A bound for the height of $H+P$.}
Let $P=(P_1,P_2)$ be a point in $E^2$.
Let $H$ be a component of the  algebraic subgroup defined by the vector $u=(\alpha,\beta)\in \Z^2$.
Let $u^{\perp}=(-\beta,\alpha)$.
Then $u^{\perp}$ defines an  algebraic subgroup $H^\perp$, and for any point $P\in E^2$ there exist two points $P_0\in H$, $P^\perp\in H^\perp$,  unique up to torsion points in $H\cap H^\perp$, such that $P=P_0+P^\perp$.
Let
\[ U = \left( \begin{array}{cc}
\alpha & \beta \\
-\beta & \alpha \\
\end{array} \right).\]
be the $2\times 2$ matrix with rows $u$ and $u^\perp$.

We remark that $u(P_0)=0$ because $P_0\in H$, and $u^\perp(P^\perp)=0$ as $P^\perp\in H^\perp$. Therefore
\begin{equation*}
UP^\perp=
\left(
\begin{array}{c}
u(P^\perp)\\
0\\
\end{array}
\right)
=
\left(
\begin{array}{c}
u(P_0+P^\perp)\\
0
\end{array}
\right)
=
\left(
\begin{array}{c}
u(P)\\
0
\end{array}
\right).
\end{equation*}
We have that $U U^{t}=U^t U=(\det U) I_2$, hence
\begin{equation*}
[\det U]P^\perp=U^tUP^\perp=U^t\left(
\begin{array}{c}
u(P)\\
0
\end{array}
\right).
\end{equation*}
Computing canonical heights we have
\begin{align*}\label{hP}
\notag(\det U)^2\hat h(P^\perp)&=\hat h([\det U]P^\perp)=\hat h\left(U^t\left(
\begin{array}{c}
u(P)\\
0
\end{array}
\right)\right)=\hat h \left(\left( \begin{array}{cc}
\alpha & -\beta \\
\beta & \alpha \\
\end{array} \right)\left(\begin{array}{c}
u(P)\\
0
\end{array}\right)\right)=\\
&=(\alpha^2+\beta^2)\hat h(u(P))=(\det U)\hat h(u(P)),
\end{align*}
and so
\begin{equation*}
   \hat h(P^\perp)=\frac{\hat h(u(P))}{\det U}.
\end{equation*}
By \cite{preprintPhilippon} we know that
\[
\hat\mu(H+P)=\hat h(P^\perp)
\]
and therefore, by Zhang's inequality \eqref{Zhangh^}
\begin{align*}
\notag h(H+P)&\leq 2(\deg H)\hat\mu(H+P)=2(\deg H)\hat h(P^\perp)= \\
\label{ultimarigacost}&=2 \frac{(\deg H)}{\det U} \hat h(u(P))=2 \frac{(\deg H)}{||{u}||^2} \hat h(u(P)).
\end{align*}
Analogously for $h_2$ using \eqref{zhangh2} and \eqref{mu2mu^}we obtain
\begin{align*}h_2(H+P)\leq & 2 (\deg H) \mu_2(H+P)\leq 2 \deg H\left(\hat \mu(H+P)+2\ccinque(E)\right)=\\&=2 \deg H\left(\frac{\hat h(u(P))}{\det U}+2\ccinque(E)\right)=2 \deg H\left(\frac{\hat h(u(P))}{||{u}||^2}+2\ccinque(E)\right).\end{align*}

By  \eqref{gradoH+P}  we get
\[
 \deg H\leq 3 ||u||^2,
\]
which leads to the bounds for $h(H+P)$ and $h_2(H+P)$ in the statement.\qedhere
\end{proof}

\subsection{Geometry of numbers}\label{geonum}
In this section  we use a classical result from the Geometry of Numbers to prove a sharp technical lemma that will be used to build  an auxiliary translate so that both its degree and height are small.

\begin{lem}\label{trapezio}
  Let $L\in\R[X_1,X_2]$ be a linear form and let $1<\kappa$. If
  \[T\geq \frac{\kappa}{\sqrt{2}(\kappa-1)^{1/4}},\] then there exists $u\in\Z^2\setminus\{(0,0)\}$ such that
  \begin{align*}
     ||{u}||&\leq T\\
   \abs{L(u)}&\leq \frac{\kappa||L||}{T},
  \end{align*}
  where    $||{u}||$ denotes the euclidean norm of $u$, $||{L}||$ the euclidean norm of the vector of the coefficients of $L$ and $|L(u)|$ is the absolute value of $L(u)$.
\end{lem}
\begin{proof}
 Let $\SSS_T\subseteq \R^2$ be the set of points $(x,y)$ satisfying the two inequalities
 \begin{align*}
  \sqrt{x^2+y^2}&\leq T\\
  \abs{L(x,y)}&\leq \kappa||L||/T.
 \end{align*}
Geometrically $\SSS_T$ is the intersection between a circle of radius $T$ and a strip of width $2\kappa/T$, as presented in the following figure (the set $\SSS_T$ is lightly shaded).
\begin{figure}[ht!]
\centering
\includegraphics[width=40mm]{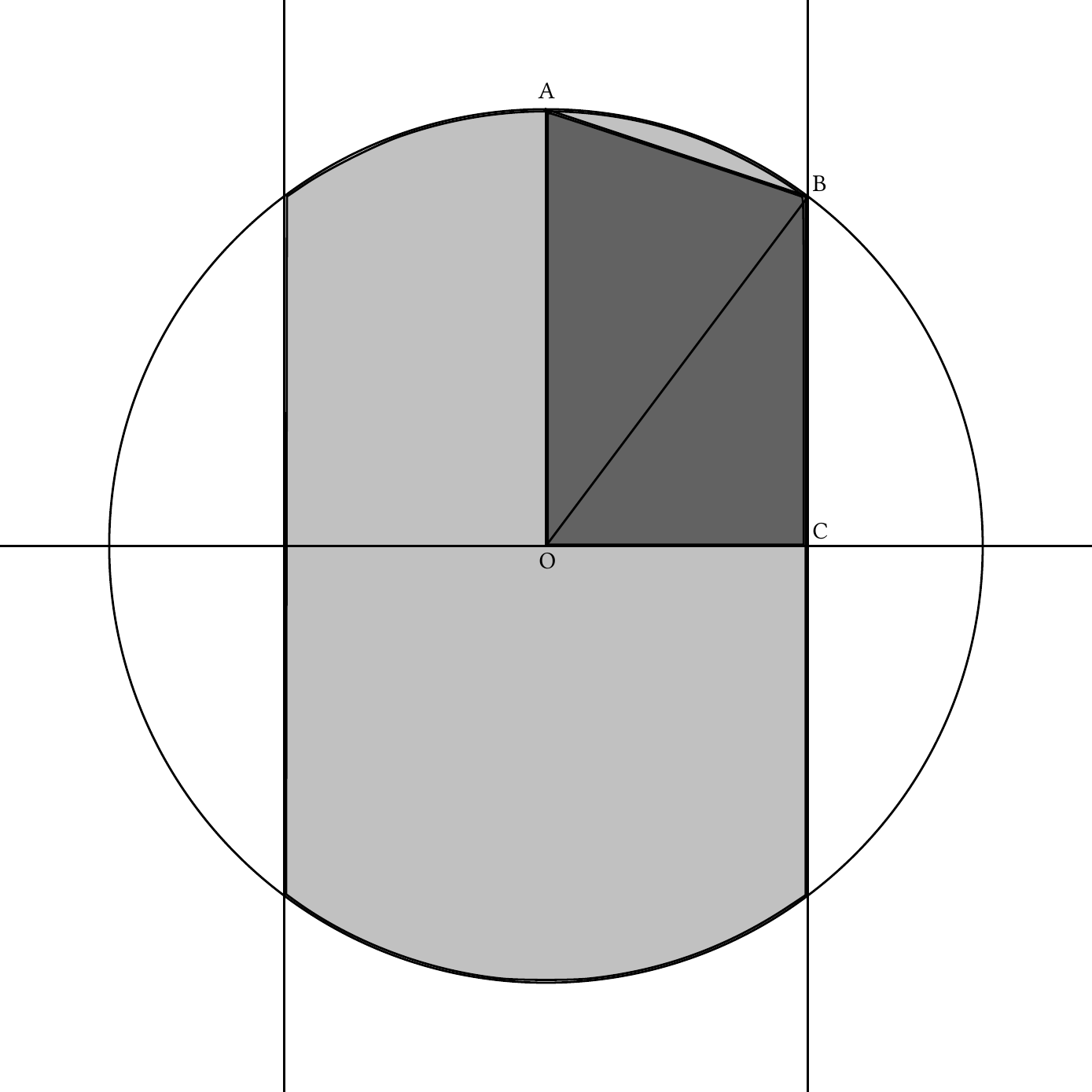}
\caption{The set $\SSS_T$}
\end{figure}

The statement of the theorem is equivalent to $\SSS_T\cap \Z^2\neq (0,0)$.
By Minkowski's Convex Body Theorem if the set $\SSS_T$ has an area bigger than 4, then the intersection  $\SSS_T\cap \Z^2$ contains points other than the origin.

The area of $\SSS_T$ is bigger than four times the area of the dark grey trapezoid in the picture, which can be easily computed as
\[\frac{\kappa}{2T}\left(T+\sqrt{T^2-\frac{\kappa^2}{T^2}}\right).\]
Therefore we need to check that
\begin{align*}
 \frac{\kappa}{2T}\left(T+\sqrt{T^2-\frac{\kappa^2}{T^2}}\right)\geq 1.
 \end{align*}
 This is trivially true for all $\kappa\geq 2$ (notice that $\frac{\kappa}{\sqrt{2}(\kappa-1)^{1/4}}\geq\sqrt{\kappa}$). If $1<\kappa<2$ an easy computation shows that the inequality holds as soon as $T\geq\frac{\kappa}{\sqrt{2}(\kappa-1)^{1/4}}$.
\end{proof}

\subsection{The auxiliary subgroup}\label{auxH}
In Proposition \ref{Prop3.1nostrocaso} we apply our Lemma \ref{trapezio} to construct  the auxiliary translate $H+P$ used in the proof  of Theorem \ref{MAINT}.
\begin{lem}\label{lem75} Let $E$ be without CM.
 Let $P=(P_1,P_2)\in E^2$ be a  point of  rank one.
Then there exists a linear form $L\in \mathbb{R}[X_1,X_2]$ with $||L||= 1$ and $$\hat h(t_1 P_1+t_2 P_2)= |L(\mathbf{t})|^2 \hat h(P)$$
for all $\mathbf{t}=(t_1,t_2)\in \mathbb{Z}^2$.
\end{lem}
\begin{proof}
Let $g$ be a generator for $\langle P_1,P_2\rangle_\mathbb{Z}$ and let $a,b\in \mathbb{Z}$ and $T_1,T_2$ torsion points   such that $P_1=[a] g+T_1$ and $P_2=[b] g+T_2$. Thus $\hat h(P)=\hat h(ag)+\hat h(bg)=(a^2+b^2)\hat h(g)$.
Consider the linear form $$L(X_1,X_2)=\frac{a X_1+b X_2}{\sqrt{a^2+b^2}}.$$
Then for all $\mathbf{t}=(t_1,t_2)\in \mathbb{Z}^2$ we have:

\begin{align*}
 \hat h(t_1 P_1+t_2 P_2)&=\hat h([at_1+bt_2]g)=(at_1+bt_2)^2 \hat h(g)=\\
 &=\frac{(a t_1+b t_2)^2}{a^2+b^2}\hat h(P) =\abs{L(\mathbf{t})}^2 \hat h(P).\qedhere
\end{align*}
\end{proof}

We can now construct the auxiliary translate.
\begin{propo}\label{Prop3.1nostrocaso} Let $E$ be  without CM. Let  $P\in  E^2$ be a point of  rank one.
Let $1<\kappa$ and $T\geq \frac{\kappa^2}{2(\kappa-1)^{1/2}}$.

Then there exists an elliptic curve $H\subseteq E^2$ such that
\begin{align*}
\deg(H+P)&\leq 3 T,\\
h_2(H+P)&\leq \frac{6\kappa^2}{T}\hat h(P)+12T\ccinque(E)
\end{align*}
 where $\ccinque(E)$ is defined in Table \ref{table1}.
\end{propo}
\begin{proof}
By Lemma \ref{lem75}, there exists a linear form $L\in\mathbb{R}[X_1,X_2]$ with $||L||=1$ such that
$\hat h(t_1 P_1+t_2 P_2)= |L(\mathbf{t})|^2 \hat h(P)$ for all vectors $t=(t_1,t_2)\in \mathbb{Z}^2$.

By Lemma \ref{trapezio}, applied to $L$, $\kappa$ and $\sqrt{T}$, there exists $u\in\mathbb{Z}^2$ such that $||u||\leq \sqrt{T}$ and $\abs{L(u)}\leq \kappa||L||/\sqrt{T}= \kappa/\sqrt{T}$.

Consider the subgroup defined by the equation $u(X)=O$ and denote by $H$ the irreducible component containing $O$.
By Proposition \ref{boundsotto}, we have that $$\deg (H+P)\leq 3||{u}||^2$$  and
$$h_2(H+P)\leq 6\hat h(u(P))+12||u||^2\ccinque(E)$$.

Combining these bounds with the above inequalities, we get that
\begin{align*}
 \deg (H+P)&\leq 3T,\\
h_2(H+P)&\leq \frac{6\kappa^2}{T}\hat h(P)+12T\ccinque(E).\qedhere
\end{align*}
\end{proof}

\subsection{Conclusion of the Proof of Theorem \ref{MAINT}}\label{SecProofMain}
 In this section  we conclude  the proof of Theorem \ref{MAINT}. We shall approximate a point of rank one with a translate constructed as above. Combing the Arithmetic B\'ezout Theorem and a good choice of  the parameters we conclude that the height of $P$ is bounded.

 \begin{proof}[Proof of Theorem \ref{MAINT}]
 If $P$ has rank zero then its height is zero and the statement is true.

 Let $T$ and $\kappa$ be real numbers with $\kappa>1$ and $\sqrt{T}\geq \kappa/\sqrt{2}(\kappa-1)^{1/4}.$
We apply Proposition \ref{Prop3.1nostrocaso} to the point $P$ of rank one, thus obtaining an elliptic curve $H$ with
\begin{align}\label{bounds}
 \deg (H+P)&\leq 3T,\\
\notag h_2(H+P)&\leq \frac{6\kappa^2}{T}\hat h(P)+12T\ccinque(E).
\end{align} The values of the free parameters $T$ and $\kappa$ will be chosen later.

We now want to bound $\hat h(P)$  in terms of $\deg (H+P)$ and $h_2(H+P)$.

Notice that the point $P$ is a component of the intersection $\Ci\cap (H+P)$, because otherwise $\Ci=H+P$,
contradicting the fact that $\Ci$ has genus $\ge2$. Therefore we can apply the Arithmetic B\'ezout Theorem to the intersection
$\Ci\cap (H+P)$, obtaining:
\begin{equation*}\label{boundh2}
h_2(P)\leq h_2(\Ci)\deg H+ h_2(H+P)\deg \Ci+\csei(1,1,8)\deg H\deg \Ci
\end{equation*}
where $\csei(1,1,8)=\frac{7}{6}(1+6\log 2)\leq 6.019$.

By Proposition \ref{prop.confronto.h2h^} we have $\hat h(P)\leq h_2(P)+2\cquattro(E)$
 so, using the bounds in formula \eqref{bounds}, we get
\begin{align*}
 \hat h(P)\leq 3 T h_2(\Ci)+\frac{6 \kappa^2}{T}\hat h(P)\deg \Ci+3 T \deg \Ci\left(4\ccinque(E)+\csei(1,1,8)\right)+2\cquattro(E).
\end{align*}

\bigskip

Let now
\begin{align*}
 \cundici(\Ci)&=6\deg \Ci,\\
 \cdue(\Ci,E)&=3 h_2(\Ci)+3\deg \Ci (4\ccinque(E)+\csei(1,1,8)), \\
 \ctre(\Ci,E)&=2\cquattro(E),
\end{align*}
so that
\begin{equation}\label{eq_ottim}
\hat h(P)\leq \cundici\frac{\kappa^2}{T}\hat h(P)+\cdue T+\ctre.
\end{equation}
We set
\begin{align*}
 \kappa&=1+\frac{1}{16\cundici^2}\\
 T&=\cundici\kappa^2\left(1+\sqrt{1+\frac{\ctre}{\cundici\cdue\kappa^2}}\right).
\end{align*}

Notice that $1<\kappa$, $T\geq \frac{\kappa^2}{2\sqrt{k-1}}$, so our assumptions on $\kappa$ and $T$ are verified.
Furthermore
\begin{equation}\label{eq_boundT}
 2\cundici\kappa^2\leq T\leq 2\cundici\kappa^2+\frac{\ctre}{2\cdue}
\end{equation}
and the coefficient of $\hat h(P)$ on the right hand side of \eqref{eq_ottim} is smaller than 1, so we can bring it to the left hand side and express $\hat h(P)$ in terms of the rest. After simplification, and using the definition of $T$, \eqref{eq_ottim} becomes
\begin{equation*}
  \hat h(P)\leq 2\cdue T +\ctre=\frac{\cdue T^2}{\cundici \kappa^2}.
\end{equation*}
Using \eqref{eq_boundT} this simplifies to
\begin{equation}\label{final}
  \hat h(P)\leq 4\cundici\cdue\kappa^2+2\ctre.
\end{equation}
After substituting everything back and noticing that  $\kappa\leq 1+\frac{1}{576}$, the last inequality \eqref{final} becomes the bound in the statement of the theorem.
\end{proof}
\begin{remark}  Theorem \ref{MAINT}' is proven in an analogous way, replacing Proposition \ref{prop.confronto.h2h^} and the constants $\ccinque(E),\cquattro(E)$ with relation \eqref{eq:1}  and the constants $d_1(E) $, $d_2(E)$.
\end{remark}

\section{Transversality and invariants for a large family of curves in $E^2$}\label{sez:criterio}
In this section we give a simple criterion to prove the transversality of a curve in $E^2$. We also show an easy argument to explicitly bound the height and the degree of a large class of curves.

\begin{lem}\label{remark.transversality}
 Let $\Ci\subseteq E^2$ be an irreducible curve. Assume that:
\begin{enumerate}[(i)]
\item\label{hyp1}  $\Ci$ is not of the form $\{P\}\times E$ or $E\times \{P\}$  for some point $P\in E$;
\item\label{hyp2} for every point $(P_1,P_2)\in\Ci$ the point $(-P_1,P_2)$ also belongs to $\Ci$.
\end{enumerate}
Then $\Ci$ is transverse.
\end{lem}

\begin{proof}
By  \eqref{hyp1}, the curve  $\Ci$ is not $\{P\}\times E$, so the natural  projection $\Ci\rightarrow E$ on the first coordinate is surjective. Thus $\Ci$ contains at least one point $(P_1,P_2)$ with $P_1$ not a torsion point in $E$. By \eqref{hyp2}, then $\Ci$ contains also the point $(-P_1,P_2)$. Observe that the only non-transverse curves in $E^2$ are translates. So
 if $\Ci$  were not transverse, then it would be a  translate  $H+Q$ of  an elliptic curve $H$ by  a point $Q=(Q_1,Q_2)$. Therefore the difference $(P_1,P_2)-(-P_1,P_2)=(2P_1,0)$ would belong to  $H$, and so would all its multiples. This implies  that  $H=E\times \{0\}$ and $\Ci=E\times \{Q_2\}$,   contradicting  \eqref{hyp1}.
\end{proof}
This last  lemma is useful to show the transversality of the following curves.
\begin{thm}\label{remark.irreducibility}
Let $E$ be defined over  a number field $k$.
 Let $E^2$ be given as in (\ref{due}) and  let $\Ci$ be the  projective closure of the curve in $E^2$ given by the additional equation
\[p(x_1)=y_2,\]
where $p(X)=p_0 X^n+p_1 X^{n-1}+\ldots+p_n$ is a non-constant polynomial in $k[X]$ of degree $n$ having $m$ coefficients different from zero.

Then $\Ci$ is transverse and its degree and normalised height are bounded as
\[\deg\Ci= 6n+9\]
and
\[h_2(\Ci)\leq 6(2n+3)\left(h_W(p)+\log m+2\cotto(E)\right)\]
where $h_W(p)=h_W(1:p_0:\ldots:p_n)$ is the height of the polynomial $p(X)$  and $\cotto(E)$ is defined in Table \ref{table1}.
\end{thm}
\begin{proof}
The transversality of $\Ci$ follows from Lemma \ref{remark.transversality}, once we have proved that $\Ci$ is irreducible.
To this aim, it is enough to check that the ideal generated by $y_1^2-x_1^3-Ax_1-B$ and $x_2^3-Ax_2+B-p(x_1)^2$ is a prime ideal in $k(x_1)[x_2,y_1]$.  This follows by observing that both polynomials are irreducible over $k(x_1)$ and involve only one of the two unknowns, with coprime exponents.
 To check the irreducibility of $x_2^3-Ax_2+B-p(x_1)^2$  we observe that a root $f(x_1)$ of this polynomial over  $k(x_1)$ gives a morphism $x_1\mapsto (f(x_1),p(x_1))$ from $\P_1$ to $E$, but such a morphism cannot exist.

 The degree of $\Ci$ is computed as an intersection product as explained in Section \ref{heightsanddegrees}.
    The preimage in $\Ci$ of a generic point of $E$ through the projection on the first component consists of 3 points. The preimage through the projection on the second component has generically $2n$ points.
 Therefore $\deg\Ci=3(2n+3)$.

We now want to estimate the height of $\Ci$.
By Zhang's inequality we have
$h_2(\Ci)\leq 2\deg\Ci\mu_2(\Ci).$
We compute an upper bound for $\mu_2(\Ci)$
by constructing an infinite set of points on $\Ci$ of bounded height.
 Let $Q_\zeta=((\zeta,y_1), (x_2,y_2))\in \Ci$, where $\zeta \in \overline{\mathbb{Q}}$ is a root of unity. Clearly there exist infinitely many such points on $\Ci$. Using the equations of $\Ci$ and classical estimates on the Weil height we have:
\[h_W(\zeta)=0,\] \[h_W(y_2)\leq h_W(1:p_0:\ldots:p_n)+\log m.\]

By  Lemma \ref{lem.confronto.h2hx}  we get:
\[ h_2(\zeta,y_1)\leq \cotto(E), \]
\[ h_2(x_2,y_2)\leq h_W(1:p_0:\ldots:p_n)+\log m+\cotto(E)\]
where $\cotto(E)$ is defined in Table \ref{table1}.
Thus for all points $Q_\zeta$ we have
\[h_2(Q_\zeta)=h_2(x_1,y_1)+h_2(\zeta,y_2)\leq h_W(1:p_0:\ldots:p_n)+\log m+2\cotto(E).\]
By the definition of essential minimum, we deduce \[\mu_2(\Ci)\leq h_W(1:p_0:\ldots:p_n)+\log m+2\cotto(E).\]
Finally, by Zhang's inequality (\ref{zhangh2}) $$h_2(\Ci)\leq 2\deg\Ci\mu_2(\Ci)\leq 6(2n+3)\left(h_W(1:p_0:\ldots:p_n)+\log m+2\cotto(E)\right)$$ as wished.
\end{proof}
 We now apply Theorem \ref{MAINT} in order to prove an effective Mordell theorem for the large family of curves defined above. The following theorem is a sharper version of Theorem \ref{polyintro} in the Introduction. If $P$ is a rational point on one of our curves, we also give bounds for the integers $a,b$ such that $P=([a]g,[b]g)$, where $g$ generates $E(k)$. These bounds are used in the algorithm in Section \ref{SecEx2} to list all the rational points and their shape explains why a $g$ with large height is advantageous for us.

  \begin{thm}\label{caso_poly}  Assume that $E$ is without CM, defined over a number field $k$ and that $E(k)$ has rank one.
Let $\Ci$ be the  projective closure of the curve given in $E^2$ by the additional equation $$p(x_1)=y_2,$$ with $p(X)\in k[X]$ a non-constant polynomial  of degree $n$ having $m$ non-zero coefficients.

 If $P\in \Ci(k)$ then
\[\hat h(P)\leq 1300.518 (2n+3)^2\left(h_W(p)+\log m+2\cotto(E)+3.01+2\ccinque(E)\right)+4\cquattro(E)\]
where $h_W(p)=h_W(1:p_0:\ldots:p_n)$ is the height of the polynomial $p(X)$ and the constants $\cotto(E)$, $\ccinque(E)$ and $\cquattro(E)$  are defined in Table \ref{table1}.

Writing $P=([a]g,[b]g)$ where $a$ and $b$ are integers and $g$ is a generator of $E(k)$ we have that
\[
 \max(\abs{a},\abs{b})\leq\left(\frac{\hat h(P)}{\hat h(g)}\right)^{1/2}.
\]
\end{thm}
\begin{proof}
Let $P\in \Ci(k)$.
 In view of Theorem \ref{remark.irreducibility}, we can apply  Theorem \ref{MAINT} to $\Ci$ in $E^2$ and,  using the bounds for $\deg\Ci$ and $h_2(\Ci)$ computed in Theorem \ref{remark.irreducibility}, we obtain the desired upper bound for $\hat h(P)$.  The bound on $\abs{a}$ and $\abs{b}$ follows from the equality $(a^2+b^2)\hat h(g)= \hat h(P)$.
\end{proof}

\section{Estimates for the  family $\Ci_n$}\label{SecEx}

 In  the  following two sections we study two special families of curves. The rough idea is to  cut a transverse curve in $E^2$ with an equation with few small integral coefficients and choosing $E$ without CM defined by a Weierstrass equation with small integral coefficients and with $E(\Q)$ of rank one.  A generator of large height can help in the implementation, but it does not play any role in the height bounds.
  Such a choice of the curve keeps the bound for the height of its rational points very small, so small that we can implement a computer search and  list  them all.

In this section  we investigate the family $\{\Ci_n\}_n$ of curves given in  Definition \ref{defCn}, i.e. cut in $E^2$ by the additional equation  $x_1^n=y_2$.

As a direct application of  Theorem  \ref{remark.irreducibility} with $p(x_1):=x_1^n$ we have:
\begin{cor}\label{gradoaltezzaCn}
For every $n\geq 1$, the curve $\Ci_n$ is transverse in $E^2$  and its degree and normalised height are bounded as
\begin{align}
\begin{split}
 \notag  \deg \Ci_n&= 6n+9,\\
 \notag  h_2(\Ci_n)&\leq 6(2n+3)\log( 3+\abs{A}+\abs{B}).
\end{split}
\end{align}
\end{cor}

Even if it is not necessary for the results of this paper, it is interesting to remark that the genus of the curves in the family $\{\Ci_n\}_n$ is unbounded  for generic  rational integers $A$ and $B$, as shown by the following lemma.
\begin{lem}\label{genus}
Suppose that the coefficients $A$ and $B$ of the elliptic curve $E$ are rational integers such that $-3A$ and $-3\Delta$ are not squares, where  $\Delta$ is the discriminant of $E$, and $B(2A^3+B^2)(3A^3+8B^2)\neq 0$. Then the curve $\Ci_n$ of Definition \ref{defCn} has genus $4n+2$.
\end{lem}
\begin{proof}
 Consider the morphism $\pi_n:\Ci_n\to\P_1$ given by the function $y_2$. The morphism $\pi_n$ has degree $6n$, because for a generic value of $y_2$ there are three possible values for $x_2$, $n$ values for $x_1$, and two values of $y_1$ for each $x_1$.

 Let $\alpha_1,\alpha_2,\alpha_3$ be the three distinct roots of the polynomial $f(T)=T^3+AT+B$;  let also $\beta_1,\beta_2,\beta_3,\beta_4$ be the roots of the polynomial $g(T)=27T^4-54BT^2+4A^3+27B^2$, which are the values such that $f(T)-\beta_i^2$ has multiple roots. If $-3A$ and  $-3\Delta$ are not squares then the polynomial $g(T)$ is irreducible over $\Q$  (\cite{BiquadraticIrreducibility}, Theorem 2); in particular, the $\beta_i$ are all distinct.

The $\beta_i$ have degree $4$ over $\Q$, and therefore they cannot be equal to any of the $\alpha_j^n$, which have degree at most $3$.
Also for all $n$ the three $\alpha_j^n$ are distinct, otherwise the ratio $\alpha_i/\alpha_j$ would be a root of 1 inside the splitting field of a polynomial of degree 3, which is easily discarded (if the ratio is 1, then $\Delta=0$, if the ratio is $-1$ then $B=0$, if the ratio is $i$ then $2A^3+B^2=0$, if the ratio is a primitive third root of unity, then $A=0$, if the ratio is a primitive sixth root of unity, then $3A^3+8B^2=0$).

The morphism $\pi_n$ is ramified over $\beta_1,\beta_2,\beta_3,\beta_4,0,\alpha_1^n,\alpha_2^n,\alpha_3^n,\infty$.
{Each of the points $\beta_i$ has} $2n$ preimages of index 2 and $2n$ unramified preimages. The point 0 has 6 preimages ramified of index $n$. The points $\alpha_i^n$ have 3 preimages ramified of index 2 and $6n-6$ unramified preimages.
The point at infinity is totally ramified.

By Hurwitz formula
\begin{align*}
2-2g(\Ci_n)&=\deg\pi_n (2-2g(\P_1))-\sum_{P\in \Ci_n}(e_P-1)\\
2-2g(\Ci_n)&=12n-(4\cdot 2n +6(n-1)+3\cdot 3 +6n -1)\\
g(\Ci_n)&=4n+2.\qedhere
\end{align*}
\end{proof}

 We remark that the five curves $E_1,\dotsc,E_5$ satisfy the hypotheses of Lemma \ref{genus}.

We now prove an effective Mordell theorem for the family $\{\Ci_n\}_n\subseteq E^2$.

The bound for the canonical height of a point $P\in\Ci_n(k)$ is a simple corollary of Theorem \ref{caso_poly}  while, for this specific family, we sharpen the bounds for the integers $a,b$ such that $P=([a]g,[b]g)$, where $g$ generates $E(k)$. This improvement speeds up the computer search.
We use here some technical height bounds proved in Section \ref{altezze}.
\begin{thm}\label{remark_Cn}
Let $E$ be an elliptic curve  defined over a number field $k$, without CM and such that $E(k)$ has rank one. Let $\{\Ci_n\}_n$ be the family of curves of Definition \ref{defCn}.
For every $n\geq 1$ and every point $P\in \Ci_n(k)$ we have
\[\hat h(P)\leq 1300.518\left(2\cotto(E)+3.01+2\ccinque(E)\right)(2n+3)^2+4\cquattro(E).\]

Writing
$P=([a]g,[b]g)$ where  $a$ and $b$ integers and $g$ is a generator of $E(k)$, we have that
\[\abs{a}\leq \left(\frac{3\hat h(P)+3\cnove(E)+6n\csette(E)}{(2n+3)\hat h(g)}\right)^{1/2}\]
and
\[|b|\leq \left(\frac{2n\hat h (P)+6 n \cuno(E)+9 \csette(E)+3\cdieci(E)}{(2n+3)\hat h(g)}\right)^{1/2}.\]
Here the constants $\ccinque(E),\dots,\cdieci(E)$ are defined in Table \ref{table1}.

\end{thm}
\begin{proof}
 From Theorem \ref{caso_poly} applied to $p(x_1):=x_1^n$  we have
\[\hat h(P)\leq 1300.518\left(2\cotto(E)+3.01+2\ccinque(E)\right)(2n+3)^2+4\cquattro(E).\]
By the  definition of $\hat{h}$ on $E^2$ and the standard properties of the N\'eron-Tate height, we have
\[
 \hat{h}(P)=\hat h([a]g)+\hat h([b]g)=(a^2+b^2)\hat{h}(g),
\]
and \begin{equation}\label{equacurva}(x([a]g))^n=y([b]g)\end{equation} because $P$ is on the curve with equation $x_1^n=y_2$.

Combining the bounds \eqref{equacurva} with \eqref{BoundSilvermanAltezze}, \eqref{hWeil}  (resp. \eqref{stima_zimmer} if $k=\Q$) and Proposition \ref{prop.confronto.h2h^},  proved in Section \ref{altezze}, we get
\begin{align*}
 \frac{2}{3}na^2\hat{h}(g)&\leq n h_W(x([a]g))+2n\csette(E)= h_W(y([b]g))+2n\csette(E)\leq \\
 &\leq h_W([b]g)+2n\csette(E)\leq h_2([b]g)+2n\csette(E)\leq \hat h([b]g)+\cnove(E)+2n\csette(E)=\\
 & \leq b^2\hat{h}(g)+\cnove(E)+2n\csette(E)
\end{align*}
where $\cnove(E)=\ccinque(E)$ in general, while if $k=\Q$ one can take $\cnove(E)=3 h_{\mathcal{W}}(E)+6\log2$.
Therefore
\[
 \frac{2n+3}{3}a^2\hat h(g)\leq \hat h(P)+\cnove(E)+2n\csette(E).
\]
which gives the bound in the statement.

Using \eqref{BoundSilvermanAltezze} and Lemma \ref{lem.confronto.h2hx},  proved in Section \ref{altezze}, we get
\begin{align*}b^2\hat h(g)&\leq \frac{3}{2} h_W(x([b]g)) +3 \csette(E)\leq  h_W(y([b]g))+\cdieci(E)+3 \csette(E)=\\
\notag &=n h_W(x([a]g))+\cdieci(E)+3 \csette(E)
\leq  \\ \notag &\leq \frac{2na^2}{3}\hat h (g)+2n \cuno(E)+\cdieci(E)+3 \csette(E)
\end{align*}
 where $\cdieci(E)=(h_W(A)+h_W(B)+\log3)/2$ and, if $k=\Q$ one can take $\cdieci(E)=\log(1+\abs{A}+\abs{B})/2$.
Therefore
\begin{align*}
\notag \frac{2n+3}{3}b^2\hat h(g)\leq  \frac{2n}{3}\hat h (P)+2n \cuno(E)+\cdieci(E)+3 \csette(E)
\end{align*}
which gives the desired bound.
\end{proof}
 We remark that the bound for $|a|$ in Theorem \ref{remark_Cn} grows like $\sqrt n$ (while the one for $|b|$ grows like $n$).

\section{Estimates for the  family $\Di_n$}
\label{Dn}
We can do similar computations for the family $\Di_n$ of Definition \ref{defDn}. Thanks to the arithmetic properties of the cyclotomic polynomials we can prove a better bounds for $h_2(\Di_n)$ than the one that follows directly from Theorem~\ref{remark.irreducibility}.

\begin{propo}\label{gradoaltezzaDn}
For every $n\geq 2$, the curve $\Di_n$ is transverse in $E^2$ and its degree and normalised height are bounded as
\begin{align}
\begin{split}
 \notag  \deg \Di_n&= 6\varphi(n)+9,\\
 \notag  h_2(\Di_n)&\leq 6(2\varphi(n)+3)\left(2^{\omega_2(n)}\log2+2\cotto(E)\right),
\end{split}
\end{align}
where $\varphi(n)$ is the Euler function, $\omega_2(n)$ is the number of distinct odd prime factors of $n$, and $\cotto(E)$ is defined in Table \ref{table1}.
\end{propo}
\begin{proof}
Transversality and the bound for the degree follow directly from Theorem~\ref{remark.irreducibility}.

Now we follow the same strategy as in the proof of Theorem~\ref{remark.irreducibility} and we construct an infinite set of points on $\Di_n$ of bounded height, getting an upper bound for $\mu_2(\Di_n)$.

 Let $Q_\zeta=((\zeta,y_1), (x_2,y_2))\in \Di_n$, where $\zeta \in \overline{\mathbb{Q}}$ is a root of unity. Clearly there exist infinitely many such points on $\Di_n$.

 We claim that for every root of unity $\zeta$ and for every $n\geq1$ we have:
 \[h_W(\Phi(\zeta))\leq 2^{\omega_2(n)}\log2,\]
 where $\omega_2(n)$ is the number of distinct odd prime factors of $n$.
 To show this, we first show that we can assume $n$ to be squarefree.

 Let $r$ be the radical part of $n$. Then we have that $\Phi_n(x)=\Phi_r(x^{n/r})$ and if $\zeta$ is a root of 1 so is $\zeta^{n/r}$.

 We can also assume $n$ to be odd, because if $n=2d$ with $d$ odd, then $\Phi_n(x)=\Phi_d(-x)$.

 Now we write
 \[\Phi_n(x)=\prod_{d\mid n}(x^d-1)^{\mu(n/d)},\]
 where $\mu(n)$ is the M{\"o}bius function, and we observe that when $n$ is odd and squarefree than there are exactly $2^{\omega_2(n)}$ factors in the product, and that $h_W(\zeta^d-1)\leq 2\log 2$ for all $\zeta$ and $d$.

 Using the equations of $\Di_n$ we have:
\[h_W(y_2)\leq 2^{\omega_2(n)}\log2.\]
Thus by Lemma \ref{lem.confronto.h2hx}
\[ h_2(\zeta,y_1)\leq \cotto(E), \quad h(x_2,y_2)\leq 2^{\omega_2(n)}\log2+\cotto(E) \]
and, using \eqref{hWeil}, for all points $Q_\zeta$ we have
\[h_2(Q_\zeta)=h_2(x_1,y_1)+h_2(\zeta,y_2)\leq 2^{\omega_2(n)}\log2+2\cotto(E).\]
By the definition of essential minimum, we deduce \[\mu_2(\Di_n)\leq 2^{\omega_2(n)}\log2+2\cotto(E).  \]
and by Zhang's inequality
$h_2(\Di_n)\leq 2\deg\Di_n\mu_2(\Di_n)$ which gives the bounds in the statement.
\end{proof}

To give an idea of the growth of the bounds above in terms of $n$, we recall that $\frac{n}{\log\log n}\ll\varphi(n)\ll n$ and that $\omega_2(n)$ has a normal value of $\log\log n$.\\

Now a direct application of Theorem~\ref{MAINT} gives the following:
\begin{cor}\label{CoroDn}
Let $E$ be an elliptic curve without CM such that $E(k)$ has rank one. Let $\{\Di_n\}_n$ be the family of curves of Definition \ref{defDn}.
For every $n\geq 2$ and every point $P\in \Di_n(k)$ we have
\[\hat h(P)\leq 1300.518\left(2^{\omega_2(n)}\log2+2\cotto(E)+3.01+2\ccinque(E)\right)(2\varphi(n)+3)^2+4\cquattro(E)\]
where the constants $\ccinque(E)$, $\cquattro(E)$ and $\cotto(E)$ are defined in Table \ref{table1}.
Writing $P=([a]g,[b]g)$ where $a$ and $b$ are integers and $g$ is a generator of $E(k)$ we have that
\[
 \max\left(\abs{a},\abs{b}\right)\leq\left(\frac{\hat h(P)}{\hat h(g)}\right)^{1/2}.
\]
\end{cor}
\begin{proof}
 The bound on $\hat h(P)$ is a direct application of  Theorem~\ref{MAINT} and the bound on $a$ and $b$ follows from Theorem \ref{caso_poly}.
\end{proof}

\section{Rational points on  explicit curves}\label{SecEx2}
 In this section we prove Theorem \ref{curveEsp} from the Introduction, which gives all the rational points of  several curves.
 The strategy here is to build many examples  by keeping fixed  the equation $$x_1^n=y_2$$  or $$ \Phi_n(x_1)=y_2$$ in $\P_2\times\P_2$ and taking many different elliptic curves $E$ in order to define the curves $\Ci_n$  and $\Di_n$ in $E^2$; see Definition \ref{defCn}.
We also recall that for $i=1,2,3,4,5$ we defined:
\begin{align*}
   E_1: y^2&=x^3+x-1\\
  \notag E_2: y^2&=x^3-26811x-7320618 \\
  \notag E_3: y^2&=x^3-675243x-213578586\\
  \notag E_4: y^2&=x^3- 110038419x + 12067837188462\\
  \notag E_5: y^2&= x^3 - 2581990371 x - 50433763600098.
\end{align*}

For these elliptic curves the discriminant and the $j$-invariant are the following:
   \begin{align}\label{DJ}
      &\Delta(E_1)=-496, 		&j(E_1)&=\frac{6912}{31}, 		\\
  \notag    &\Delta(E_2)=-21918062700048384, 	&j(E_2)&=-\frac{979146657}{10069019}, 	\\
  \notag    &\Delta(E_3)=-1765662163329024,	&j(E_3)&=-\frac{15641881075729}{811134},\\
   \notag   &\Delta(E_4)=-62828050697723854898526892032, &j(E_4)&=-\frac{2507136440062325499}{1068992890181390681},\\
   \notag   &\Delta(E_5)=2830613675881894730558078976, &j(E_5)&=\frac{874525671242290400569417}{1300365970941935616}.
   \end{align}
We recall that all CM elliptic curves have an integral $j$-invariant; this shows that the curves $E_i$ are without CM for $i=1,\ldots,5$.

 Using a databases of elliptic curve data such as \cite{Cremona} or~\cite{lmfdb}, we checked that for every $i\neq 2$, $E_i$ has no torsion points defined over $\Q$ and that $E_i(\Q)$ has rank one. We also found in the tables an explicit generator $g_i$ for $E_i(\Q)$ and we computed  $\hat h(g_i)$ using  the function \texttt{ellheight} of PARI/GP \cite{PARI} (notice that the canonical height of PARI/GP is two thirds of ours).
  A generator for the curve $E_2$, which has a conductor too big to appear in Cremona's tables, was given in \cite{Sil}, Example 3.
Collecting these informations we have that the generators of $E_i(\Q)$ are:
\begin{align*}
   g_1&=(1,1), \\
 \notag  g_2&=\left(\frac{290083549425751}{23921262225},\frac{4940195839487330160124}{3699782022029625}\right), \\
 \notag  g_3&=\left(\frac{930273}{484},-\frac{796052583}{10648}\right),	\\
 \notag  g_4&=\left(\frac{3228005993902971489}{128791448271424},\frac{7316042869129182048724448529}{1461606751179427091968}\right),\\
 \notag  g_5&=\left(\frac{-9750023890880795040300239250862047101114}{335283704622805743122062106485469025},\right.\\
   &\left.\frac{47202993140158532858227353349489655613892905428267026719866}{194141629146024723477365694402532030141467059091092625}\right).
\end{align*}
where
\begin{align}\label{hgen}
&\hat h(g_1)\geq 0.377, &\hat h(g_2)\geq 47.888, &&\hat h(g_3)\geq 17.649, \\
\notag &\hat h(g_4)\geq 60.674, &\hat h(g_5)\geq 136.823. &
\end{align}
We can now state our bounds for the 5 families of curves $\{\Ci_n\}_n$ in $E_i^2$.
\begin{thm}\label{teoEsempi}
Let $P\in \Ci_n(\Q)\subseteq E^2$ where $E$ is one of the curves  $E_i$ for $i=1,\ldots,5$. We write $P$ in terms of the generator $g_i$ as $P=([a]g_i,[b]g_i)$.
Then
\begin{enumerate}
   \item If $E=E_1$ we have
   \begin{align*}
      \hat h(P)&\leq 73027 \cdot n^2+219081\cdot n +164320,\\
      \abs{a}&\leq \left(\frac{581115 \cdot n^2+1743376\cdot n+1307618}{2n+3}\right)^{1/2},\\
      \abs{b}&\leq \left(\frac{387410\cdot n^3 + 1162229\cdot n^2 + 871760 \cdot n + 54}{2n+3}\right)^{1/2}.
   \end{align*}
   \item If $E=E_2$ we have
   \begin{align*}
   \hat h(P)&\leq 311345 \cdot n^2+934033\cdot n+700566,\\
   \abs{a}&\leq \left(\frac{19505\cdot n^2+58515\cdot n+43889}{2n+3}\right)^{1/2},\\
   \abs{b}&\leq \left(\frac{13004\cdot n^3+39010\cdot n^2+29260\cdot n+2}{2n+3}\right)^{1/2}.
   \end{align*}
   \item If $E=E_3$ we have
   \begin{align*}
      \hat h(P)&\leq 373925\cdot n^2+1121775\cdot n+841382,\\
      \abs{a}&\leq \left(\frac{63561\cdot n^2+190683\cdot n+143021}{2n+3}\right)^{1/2},\\
      \abs{b}&\leq \left(\frac{42374 \cdot n^3+127121\cdot n^2+95349\cdot n+5}{2n+3}\right)^{1/2}.
   \end{align*}
   \item If $E=E_4$ we have
   \begin{align*}
      \hat h(P)&\leq 534732\cdot n^2 + 1604195\cdot n + 1203216,\\
      \abs{a}&\leq \left(\frac{26440 \cdot n^2 + 79320\cdot n +   59494}{2n+3}\right)^{1/2},\\
      \abs{b}&\leq \left(\frac{17627 \cdot n^3 + 52880 \cdot n^2 +  39663 \cdot n +2}{2n+3}\right)^{1/2}.
   \end{align*}
   \item If $E=E_5$ we have
   \begin{align*}
      \hat h(P)&\leq 566995\cdot n^2 + 1700984\cdot n +  1275813,\\
      \abs{a}&\leq \left(\frac{12433\cdot n^2 + 37297\cdot n +  27974}{2n+3}\right)^{1/2},\\
      \abs{b}&\leq \left(\frac{8289\cdot n^3 + 24865\cdot n^2 +18650\cdot n +1}{2n+3}\right)^{1/2}.
   \end{align*}
\end{enumerate}
\end{thm}
\begin{proof}
   The proof is an application of Theorem \ref{remark_Cn}. First, we need to compute all the invariants intervening in the bounds. Notice that $\deg \Ci_n$, $h_2(\Ci_n)$ are bounded in Corollary \ref{gradoaltezzaCn}, while $\Delta(E_i)$ and $j(E_i)$ are bounded in \eqref{DJ} and a lower bound for $\hat h(g_i)$ is given in \eqref{hgen}.

We are left to estimate
 $h_{\mathcal{W}}(E_i)=h_W(1: A^{1/2}_i:B_i^{1/3})$ as defined in \eqref{hWeierE}. We obtain:
   \begin{align*}
      &h_\mathcal{W}(E_1)=0,
      &h_{\mathcal W}(E_2)\leq 5.269,
      &&h_\mathcal{W}(E_3)\leq 6.712,\\
      &h_\mathcal{W}(E_4)\leq 10.041,
      &h_\mathcal{W}(E_5)\leq 10.836,
   \end{align*}
 In addition, by Table \ref{table1} we get:
  \begin{align*}
     &\ccinque(E_1)\leq 4.709, 	&\cquattro(E_1)&\leq 2.423,	&\csette(E_1)&\leq 2.037, 	&\cuno(E_1)&\leq 2.31,\\
     &\ccinque(E_2)\leq 20.515, &\cquattro(E_2)&\leq 10.33,	&\csette(E_2)&\leq 4.587, 	&\cuno(E_2)&\leq 5.353,\\
     &\ccinque(E_3)\leq 24.843, &\cquattro(E_3)&\leq 12.494,	&\csette(E_3)& \leq 5.394, 	&\cuno(E_3)&\leq 6.563,\\
     &\ccinque(E_4)\leq 34.83,  &\cquattro(E_4)&\leq 17.487,	&\csette(E_4)&\leq 6.667, 	&\cuno(E_4)&\leq 8.336, \\
     &\ccinque(E_5)\leq 37.216, &\cquattro(E_5)&\leq 18.68,	&\csette(E_5)&\leq 7.456, 	&\cuno(E_5)&\leq 9.656.
   \end{align*}
and
 \begin{align*}
     &\cnove(E_1)\leq 4.159,	&\cotto(E_1)&\leq 0.805,   &\cdieci(E_1)&\leq 0.55,\\
     &\cnove(E_2)\leq 9.428,	&\cotto(E_2)&\leq 7.905,   &\cdieci(E_2)&\leq 7.904,\\
     &\cnove(E_3)\leq 10.871,	&\cotto(E_3)&\leq 9.592,    &\cdieci(E_3)&\leq 9.592,\\
     &\cnove(E_4)\leq 14.2,	&\cotto(E_4)&\leq 15.061,    &\cdieci(E_4)&\leq 15.061, \\
     &\cnove(E_5)\leq 14.995, 	&\cotto(E_5)&\leq 15.776,    &\cdieci(E_5)&\leq 15.776.
   \end{align*}

We can now replace all the above values  in the formulas of Theorem \ref{remark_Cn} and obtain the bounds in our statement.
\end{proof}

 We have an analogous result for the 5 families curves $\Di_n$ in $E_i^2$, which we write for simplicity for the subfamilies consisting of all elements for which the index $n$ is a prime.
\begin{thm}\label{teoEsempiDn}
Let $P\in \Di_n(\Q)\subseteq E^2$ where $E$ is one of the curves  $E_i$ for $i=1,\ldots,5$. We write $P$ in terms of the generator $g_i$ as $P=([a]g_i,[b]g_i)$. Assume that $n$ is a prime number.
Then
\begin{enumerate}
   \item If $E=E_1$ we have
   \begin{align*}
      \hat h(P)&\leq 80239 n^2+80239 n+20070,\\
     \max\left( \abs{a},\abs{b}\right)&\leq \sqrt{212834 n^2+212834 n+53235}.
   \end{align*}
   \item If $E=E_2$ we have
   \begin{align*}
   \hat h(P)&\leq 318556 n^2 +318556 n+79681 ,\\
  \max\left( \abs{a},\abs{b}\right)&\leq \sqrt{6653 n^2+6653 n+1664}.
   \end{align*}
   \item If $E=E_3$ we have
   \begin{align*}
      \hat h(P)&\leq 381137 n^2+381137 n+95335,\\
      \max\left( \abs{a},\abs{b}\right)&\leq \sqrt{21596 n^2+21596 n+5401}.
   \end{align*}
   \item If $E=E_4$ we have
   \begin{align*}
      \hat h(P)&\leq 541943 n^2+541943 n+135556,\\
     \max\left( \abs{a},\abs{b}\right)&\leq \sqrt{8933 n^2+8933 n+2235 }.
   \end{align*}
   \item If $E=E_5$ we have
   \begin{align*}
      \hat h(P)&\leq 574207 n^2+574207 n+143627,\\
     \max\left( \abs{a},\abs{b}\right)&\leq \sqrt{4197 n^2+4197 n+1050} .
   \end{align*}
\end{enumerate}
\end{thm}
\begin{proof}
 These bounds are a direct application of Corollary~\ref{CoroDn}. The relevant numerical constants  are already listed in the proof of Theorem~\ref{teoEsempi}.
\end{proof}

With these sharp estimates we are ready to implement the computer search up to the computed bounds for the rational points on our curves, and so to prove  Theorem \ref{curveEsp}.

 To perform the computer search, we used the PARI/GP \cite{PARI} computer algebra system, an open source program
freely available at \url{http://pari.math.u-bordeaux.fr}

We first tried to implement a naive algorithm that performs the multiples of the points $g_i$ on the elliptic curve using PARI's implementation of the exact arithmetic of the elliptic curve over the rationals. This has proved far too time-consuming and was only done for $n=1$.

Then we used a more efficient algorithm pointed out by Joseph H. Silverman. The idea is to identify the elliptic curve $E$ with a quotient $\C/\Lambda$ and see the multiplication by $a$ on $E$ as induced by the multiplication by $a$ in $\C$.
This algorithm is quite fast and capable of performing the computations up to about $n=50$.

The algorithm that we used in our final computation  is due to K. Belabas and uses a sieving technique.
 It is  very general and it can be applied to any of the curves of Theorem \ref{caso_poly} when $k=\Q$, although we performed the computations only for curves belonging to the families $\Ci_n$ and $\Di_n$.

The idea is that, in order to test which of a finite but very big number of points actually lie on the curve, we test when this happens modulo many big primes.

We are very thankful to K. Belabas for  providing us the sieving algorithm presented in the following proof.

\begin{proof}[Proof of Theorem \ref{curveEsp}]
Theorem \ref{curveEsp}  is now a consequence of Theorem \ref{teoEsempi} and Theorem \ref{teoEsempiDn} and an extensive computer search.

For each of the curves $E_i$ and for each $n$, Theorem \ref{teoEsempi} gives us upper bounds for the integers $a,b$ such that $([a]g_i,[b]g_i)\in\Ci_n$, therefore we only need to check which of finitely many points lie on the curve $\Ci_n$ (resp. $\Di_n$).

Even though, as remarked in the Introduction, the computations for large $n$ are superseded by the results in Section~\ref{S:examples} of the appendix, we think it is worthwhile, for future applications, to give some details on how they were performed.   In particular, we present here the  PARI code  used to implement Belabas' algorithm in the general case for curves $\Ci$ as in Theorem \ref{caso_poly}, cut in $E^2$ by the additional equation $p(x_1)=y_2$, with $p(X)$ a polynomial in $\Z[X]$.  The algorithm can possibly  be adapted to curves of different shapes.

We fix  the polynomial $p(X)$, called  \texttt{Pol(X)} in the code, of degree $n$ and we start by initialising the following variables
\begin{verbatim}
A,B,Ba,g,ntest
\end{verbatim}
where \texttt{A} and \texttt{B} are the coefficients of the Weierstrass model of $E$,  \texttt{Ba} is the ceiling of the bound on $\abs{a}$ obtained for the chosen polynomial $p(X)$, \texttt{g} is the generator of $E(\Q)$ and \texttt{ntest} is a parameter used to decide when to stop the sieving process.

Then we define the following program, that we indent here for readability
\begin{verbatim}
0   E = ellinit([A,B]);
1   D=abs(E.disc);
2   Sievea() =
3   {
4     p = nextprime(Ba);
5     L = [1..Ba];
6     cnt = 1;
7     while(1,
8       if(D%p==0,next);
9       if(denominator(g[1])%p==0,next);
10      oldnL = #L;
11      ag = [0];
12      Ep = ellinit(E, p);
13      Lp = List([]);
14      for (a = 1, Ba,
15        ag = elladd(Ep, ag, g);
16        if (#ag == 1, listput(Lp, a); next);
17        x = ag[1];
18        xp = Mod(x,p);
19        if(polrootsmod('X^3 + A*'X + B - Pol(xp)^2, p), listput(Lp, a)) ;
20      );
21      listsort(Lp);
22      L = setintersect(L, Vec(Lp));
23      if (#L == oldnL, cnt++, cnt = 0);
24      if (#L == 0 || cnt > ntest, break);
25      p = nextprime(p+1);
26    );
27    printf("L=%s\n",L);
28  }
\end{verbatim}
The core of the algorithm is the \texttt{while} loop in line 7. This loop iterates over the prime \texttt{p}, which is initialised  in line 3 to a value bigger than \texttt{Ba}. At each iteration the algorithm takes the list \texttt{L}, which initially contains all positive values of $a$ up to the bound  \texttt{Ba}, and checks for which of these values there exists a point $([a]g,[b]g)$ on the curve $\Ci_n$ reduced modulo \texttt{p}. This check is done in the \texttt{for} loop at line 13. The $a$ that correspond to points modulo \texttt{p} are stored in the list \texttt{Lp} and the values of $a$ that do not correspond to a point are removed from the list \texttt{L} at line 21. The algorithm then changes the prime number \texttt{p} to the next one, and the loop starts again.
The check at  lines 8 and 9 ensures that the primes of bad reduction for the curve \texttt{E} and those that divide the denominator of the generator are discarded.
The algorithm keeps sieving through the list \texttt{L} until either the list becomes empty, which proves that there are no rational points, or \texttt{ntest} iterations pass without any value of $a$ being discarded. When this happens the program outputs these values of $a$, which are candidate solutions and need to be investigated further.

 In our explicit examples we found that setting \texttt{ntest} to  25  was enough, and no candidate solution was ever found other than those arising from rational points on $E_1\times E_1$.
\end{proof}

 The variable \texttt{Ba}, and hence the length of the list \texttt{L} in line 5, is directly proportional to the square root of the height of the coefficients of the Weierstrass model of $E$ and inversely proportional to the square root of the
height of the generator of $E(\Q)$, which explains the speed improvement when the generator has a big height compared to the coefficients.

We remark that with a simple modification this algorithm can be made deterministic by stopping the iteration in a suitably-chosen way depending on the degree and the coefficients of the curve. However this increases, in general, the running time compared to a good heuristic choice of the parameter \texttt{ntest}.

 When adapting the algorithm to other examples,  if for a certain choice of \texttt{ntest} the above algorithm returns a list of possible values, one can either increase \texttt{ntest} or directly check the values with the floating point algorithm.\\

 We finally notice that for our method it is not necessary to know  {\em a priori}  a generator $g$ of $E(\Q)$.
 Indeed  we can argue  as follows. Theorem \ref{caso_poly} gives the bound $\hat{h}(P)\le D$ for  any  rational point on $\Ci$. Thus we only need to search for a generator $g$ of $E(\Q)$ such that  $\hat{h}(g)\leq\hat{h}(P)$, otherwise $\Ci(\Q)$ is trivially empty.   To this purpose, one can use a suitable search algorithm for generators  of height at most $D$ on elliptic curves of rank one, as described in \cite{Sil}. For instance with Silverman's Canonical Height Search Algorithm finding a generator of $E(\Q)$ takes about $O(\sqrt{N_E}+D)$, where $N_E$ is the conductor of $E$. This is also one of the few algorithms that can deal with curves of high conductor.

\section*{Acknowledgments}
We are indebted to K. Belabas for writing the algorithm to conclude the proof of Theorem \ref{curveEsp} and for his kind answers on some technical aspects of PARI/GP.  We are thankful to M. Stoll for his  useful remarks  which helped us to improve the paper and for his nice appendix. We warmly thank J. H. Silverman for his useful suggestions and for his interest in our work. We are grateful to P. Philippon for answering some questions on the comparison of several height functions. We also thank \"O. Imamoglu for her comments on an earlier version of this paper. S. Checcoli's work has been funded by the ANR project Gardio 14-CE25-0015. E. Viada thanks the FNS (Fonds National Suisse) Project PP00P2-123262/1 for the financial support.

\medskip

\noindent Sara Checcoli:
 Institut Fourier,
100 rue des Maths,
BP74 38402 Saint-Martin-d'H\`eres Cedex, France.
email: sara.checcoli@ujf-grenoble.fr
\medskip\\
Francesco Veneziano:
Mathematisches Institut,
Universit\"{a}t Basel,
Spiegelgasse 1,
CH-4051 Basel,
Switzerland.
email: francesco.veneziano@unibas.ch
\medskip\\
Evelina Viada:
ETH Zurich, R\"amistrasse 101, 8092, Z\"urich, Switzerland and
Mathematisches Institut,
Georg-August-Universit\"at,
Bunsenstra\ss e 3-5,
D-D-37073, G\"ottingen,
Germany.
email: evelina.viada@math.ethz.ch.


\newpage
\appendix
\section{}\label{appendix}
\begin{center}\author{by M. Stoll}\end{center}

\bigskip

As mentioned
in the introduction, the approach taken in the main paper applies
in basically the same setting as Demjanenko's method.
The first goal of this appendix is to provide a comparison between
the two approaches, first in general terms, and then more concretely
for a family of curves of genus~$2$ to which Demjanenko's approach
can be applied quite easily.

In the main paper,
the bound obtained is used to find explicitly the set of rational points
on certain curves $\calC_n(E)$ and~$\calD_n(E)$ sitting in~$E \times E$
for certain elliptic curves~$E$, where the parameter~$n$
ranges up to an upper bound depending on~$E$.
The second goal of this appendix is to complete the analysis of these examples
by determining the set of rational points on the
curves $\calC_n(E)$ and~$\calD_n(E)$
(for the five curves~$E$ considered there) for \emph{all}~$n$.
The additional ingredient we use is an analysis of the $\ell$-adic
behaviour of points on the curves close to the origin. This analysis
leads to a fast-growing lower bound for the height of a
point $(P_1, P_2) \in \calC(\Q)$ that is not the origin $(O, O)$ and is also
not a pair of integral points on~$E$. Since this lower bound grows
faster than the upper bound, this implies that all rational points
on~$\calC$ distinct from $(O, O)$ must be pairs of integral points
as soon as $n$ is large enough. Since the number of integral points
on~$E$ is finite, this result shows that $\calC_n(E)(\Q)$ and $\calD_n(E)(\Q)$
are contained in a fixed finite set for all sufficiently large~$n$.
It is then an easy matter to determine which of these finitely many
points are on which of the curves. This approach can be used more
generally when the curve~$\calC$ is given by an equation of the form
$F_1(x_1,y_1) = F_2(x_2,y_2)$ with polynomials $F_1, F_2$ such
that the degrees of $F_1(x,y)$ and~$F_2(x,y)$, considered as rational
functions on~$E$, differ. If the ratio
of the degrees is sufficiently large compared to the height and degree
of~$\calC$, then all rational points on~$\calC$ distinct from~$(O, O)$ must be pairs
of $S$-integral points on~$E$ (for an explicit finite set~$S$ of primes),
of which there are only finitely many.


\subsection{Comparison with Demjanenko's method} \label{App:comparison} \strut

The setting of Demjanenko's method is a curve~$\calC$, which we take to
be defined over~$\Qbar$, that allows $N$ independent morphisms
$\phi_j \colon \calC \to E$, $j = 1,2,\ldots,N$, to a fixed elliptic
curve~$E$ also defined over~$\Qbar$. ``Independent'' here means
that no nontrivial integral linear combination of the~$\phi_j$
is constant. This is equivalent to saying that the image
of~$\calC$ in~$E^N$ under the product of the~$\phi_j$ is transverse,
and so this setting is essentially the same as considering
a transverse curve in~$E^N$ as is done in the main paper.

We now paraphrase Demjanenko's method~\cite{Demj} in the case $N = 2$
as applied in~\cites{KuleszApplicazioneMD,Kul2,Kul3}.
The description below is close to Silverman's in~\cite{Silverman1987}.
Consider a curve~$\calC$ (of genus $\ge 2$) over~$\Qbar$
with two independent morphisms $\phi_1, \phi_2 \colon \calC \to E$ to
an elliptic curve~$E$ defined over~$\Qbar$.
The independence of the morphisms implies that the quadratic form (in~$\alpha_1, \alpha_2$)
$\deg(\alpha_1 \phi_1 + \alpha_2 \phi_2)$ is positive definite.
Fix a height~$h$ on~$\calC$, which is scaled so that
$\hat{h}(\phi_j(P)) = (\deg \phi_j + o(1)) h(P)$ for $P \in \calC(\Qbar)$
as $h(P) \to \infty$.
Then there are constants~$c_j$
such that for all $P \in \calC(\Qbar)$ with $h(P) \ge 1$
(see~\cite{HS}*{Theorem~B.5.9})
\begin{align*}
    \bigl|(\deg \phi_1) h(P) - \hat{h}(\phi_1(P))\bigr| &\le c_1 \sqrt{h(P)}, \\
    \bigl|(\deg \phi_2) h(P) - \hat{h}(\phi_2(P))\bigr| &\le c_2 \sqrt{h(P)}, \\
    \bigl|(\deg (\phi_1 + \phi_2)) h(P) - \hat{h}(\phi_1(P) + \phi_2(P))\bigr| &\le c_3 \sqrt{h(P)}.
\end{align*}
We write $\langle P_1, P_2 \rangle = \tfrac{1}{2}\bigl(\hat{h}(P_1+P_2) - \hat{h}(P_1) - \hat{h}(P_2)\bigr)$
for the height pairing and similarly
$\langle \phi_1, \phi_2 \rangle = \tfrac{1}{2}\bigl(\deg(\phi_1 + \phi_2) - \deg \phi_1 - \deg \phi_2\bigr)$.
Then we deduce that
\[ \bigl|\langle \phi_1, \phi_2 \rangle h(P) - \langle \phi_1(P), \phi_2(P) \rangle\bigr| \le c_4 \sqrt{h(P)} \]
with $c_4 = \tfrac{1}{2}(c_1 + c_2 + c_3)$.
This gives that
\[ \deg(\alpha_1 \phi_2 + \alpha_2 \phi_2) h(P)
          - \hat{h}\bigl(\alpha_1 \phi_1(P) + \alpha_2 \phi_2(P)\bigr)
    \le \bigl(\alpha_1^2 c_1 + 2 |\alpha_1 \alpha_2| c_4 + \alpha_2^2 c_2\bigr) \sqrt{h(P)}
\]
and so (still for $h(P) \ge 1$)
\begin{equation} \label{E:boundD}
  h(P) \le \frac{\hat{h}\bigl(\alpha_1 \phi_1(P) + \alpha_2 \phi_2(P)\bigr)}%
                 {\deg(\alpha_1 \phi_2 + \alpha_2 \phi_2)} + \gamma(\alpha_1, \alpha_2) \sqrt{h(P)},
\end{equation}
where
\[ \gamma(\alpha_1, \alpha_2)
     = \frac{\alpha_1^2 c_1 + 2 |\alpha_1 \alpha_2| c_4 + \alpha_2^2 c_2}%
             {\alpha_1^2 (\deg \phi_1) + 2 \alpha_1 \alpha_2 \langle \phi_1, \phi_2 \rangle
                  + \alpha_2^2 (\deg \phi_2)}.
\]
Since the denominator is positive definite, there is a uniform upper bound,
for example
\[ \gamma(\alpha_1, \alpha_2) \le \gamma := \frac{2 \max\{c_1, c_2\} + \tfrac{1}{2} c_3}{\lambda}, \]
where $\lambda$ is the smaller eigenvalue of the
matrix~$\bigl(\langle \phi_i, \phi_j \rangle\bigr)_{1 \le i, j \le 2}$.

Now let $P \in \calC(\Qbar)$ be such that $\phi_1(P)$ and~$\phi_2(P)$
generate a subgroup of rank~$1$ in~$E$. Then there are $\alpha_1, \alpha_2 \in \Z$,
not both zero, such that $\alpha_1 \phi_1(P) + \alpha_2 \phi_2(P) = O$.
Then from~\eqref{E:boundD} we obtain the bound $h(P) \le \max\{1, \gamma^2\}$.
In particular, if $\calC$, $E$ and the morphisms are defined over some
number field~$K$ and $E(K)$ has rank~$1$, then $h(P) \le \max\{1, \gamma^2\}$ for
all $K$-rational points~$P$ on~$\calC$. (For this application it is
sufficient to use bounds $c_j$ that are only valid for $K$-rational points.)

We get a better bound when (writing $\phi_3 = \phi_1 + \phi_2$) suitable
positive multiples of the pulled-back divisors $\phi_j^*(O)$ are linearly
equivalent, for $j = 1,2,3$. We can then take the height~$h$ so that
\[ (\deg \phi_j) h(P) = 3 h_{\phi_j^*(O)}(P) + O(1)
                      = 3 h_O(\phi_j(P)) + O(1)
                      = \hat{h}(\phi_j(P)) + O(1), \]
where $h_D$ denotes a height associated to the divisor~$D$, compare~\cite{HS}*{Theorem~B.3.2}.
We then obtain bounds as above, but without the $\sqrt{h(P)}$ term.
The final bound is then just $h(P) \le \gamma$.

One situation where this applies is when $\calC$ is hyperelliptic.
In this case, after translation by a constant point in~$E$,
any morphism $\phi \colon \calC \to E$ descends to a morphism
$\tilde{\phi} \colon \PP^1 \to \PP^1$ on $x$-coordinates,
so that we have a commutative diagram
\[ \xymatrix{\calC \ar[d]_{\pi_{\calC}} \ar[r]^{\phi} & E \ar[d]^{\pi_E} \\
             \PP^1 \ar[r]^{\tilde{\phi}} & \PP^1}
\]
where $\pi_{\calC}$ and~$\pi_E$ are the $x$-coordinate morphisms.
Then
\[ 2 \phi^*(O) = \phi^*(2O) = \phi^* \pi_E^*(\infty)
               = \pi_{\calC}^* \tilde{\phi}^*(\infty)
               \sim (\deg \phi) \pi_{\calC}^*(\infty)
\]
and so $2 \phi^*(O)$ is linearly equivalent to a multiple of
$\pi_{\calC}^*(\infty)$ for every~$\phi$.

We can expect $c_j$ to be be of the order of $(\deg \phi_j) h(\calC)$
(with $\phi_3 = \phi_1 + \phi_2$) with some notion of height for~$\calC$.
The resulting height bound will then have order of magnitude~$h(\calC)$
in the special case just discussed (the contribution
of the degrees will cancel, since the degrees also occur in the
denominator of~$\gamma(\alpha,\beta)$).
This will usually be better than the bound obtained in the main paper;
see for example the comparison in Section~\ref{S:genus2} below.
In the general case, we obtain a bound that
has order of magnitude~$h(\calC)^2$; this is to be compared
with $\deg(\calC)\bigl(h(\calC) + \deg(\calC)\bigr)$ for the bound obtained in the
main paper (which likely has a larger constant in front).

If one starts with a concrete curve~$\calC$ with two morphisms to~$E$,
then it will usually not be very hard to find the constants needed
to get a bound as derived in this section, in particular when
$\calC$ is hyperelliptic.
On the other hand, starting from a curve $\calC$ given as a subvariety
of $E \times E$ by some equation, one first has to fix a suitable
height on~$\calC$. It appears natural to take the height used previously,
namely $\hat{h}(P_1) + \hat{h}(P_2)$, suitably scaled, which
means that we divide by the sum of the degrees of the two morphisms to~$E$.
We then have to bound
\[ (\deg \phi_2) \hat{h}(P_1) - (\deg \phi_1) \hat{h}(P_2)
    \quad\text{and (say)}\quad
   (\deg (\phi_1+\phi_2)) \hat{h}(P_1) - (\deg \phi_1) \hat{h}(P_1 + P_2)
\]
to obtain the necessary constants. This may be not so easy in general.
So in this situation, the method of Checcoli, Veneziano and Viada
produces a bound that is easy
to compute, but is likely to be larger than what we would obtain
from Demjanenko's method. One possible source for the comparative
weakness of the bound is that the Arithmetic B\'ezout Theorem
bounds the \emph{sum} of the heights of \emph{all} points in the fibre of
$\alpha_1 \phi_1 + \alpha_2 \phi_2 \colon \calC \to E$ that contains~$P$,
and this sum (with potentially many terms) is used to bound a
single summand.


\subsection{An application to curves of genus 2} \label{S:genus2} \strut

We illustrate the comparison between the two approaches
by considering a family of curves of genus~2
whose members have two independent morphisms to the same elliptic curve.
This is a setting where Demjanenko's method can be applied fairly
easily (this has been done in~\cite{Kul3}) and with constant height
difference bounds, which gives Demjanenko's approach a considerable advantage.

A curve of genus~$2$ over~$\Q$ is given by an affine equation
\[ \calC \colon y^2 = f_6 x^6 + f_5 x^5 + \ldots + f_1 x + f_0 \]
with $f_0, \ldots, f_6 \in \Z$ and such that the right hand
side has degree at least~$5$ and has no multiple roots.
Assume that $\calC$ has two morphisms $\pi_1$, $\pi_2$
to the same elliptic curve~$E$. The simplest case is when both $\pi_1$
and~$\pi_2$ have degree~$2$. If $\calC$ is a double cover of~$E$, then $\calC$
has an `extra involution'~$\sigma$, which is an automorphism of order~$2$ that is not
the hyperelliptic involution~$\iota$. One can check that in this situation
$\sigma$ has two fixed points with the same $x$-coordinate, and the same
is true for~$\sigma\iota$. (The other possibility would be that $\sigma$
and~$\sigma\iota$ have the same two Weierstrass points as fixed points,
but this would force $\sigma$ to be of order~$4$.) These two $x$-coordinates
are then rational (we assume that $\calC \to E$ and hence~$\sigma$ is defined
over~$\Q$), and so we can assume that they are $0$ and~$\infty$; then
$\sigma$ is given by $(x,y) \mapsto (-x,y)$ and $\sigma\iota$
is $(x,y) \mapsto (-x,-y)$. The equation of~$\calC$ then has the form
\[ y^2 = a x^6 + b x^4 + c x^2 + d \]
and the quotient elliptic curve $\calC/\langle \sigma \rangle$
is $E_1 \colon y^2 = x^3 + b x^2 + a c x + a^2 d$, whereas the
quotient $\calC/\langle \sigma\iota \rangle$ is
$E_2 \colon y^2 = x^3 + c x^2 + d b x + d^2 a$.
In the simplest situation, $E_2 = E_1$, so $b = c$ and $a = d$.
(In general, $E_2$ can be isomorphic to~$E_1$ without being equal to it.)
So we now consider the curve
\[ \calC \colon y^2 = a x^6 + b x^4 + b x^2 + a, \]
where $a, b \in \Z$. We assume that $a \neq 0, -b, b/3$ to ensure that
$\calC$ has genus~$2$. A  Weierstrass equation for $E = E_1 = E_2$ is
\[ y^2 = x^3 + b x^2 + a b x + a^3. \]
To apply the results of the main paper, we transform this into the
short Weierstrass equation
\[ E \colon y^2 = x^3 + 27 b (3 a - b) x + 27 (27 a^3 - 9 a b^2 + 2 b^3). \]
(We remark that this increases the height of the equation defining $E$,
which leads to a final bound that is worse than what could be obtained
by working with the `long' equation directly.)
We can then embed $\calC \injects E \times E$ via
\[ (x,y) \longmapsto \bigl((9 a x^2 + 3 b, 27 a y), (9 a x^{-2} + 3 b, 27 a x^{-3} y)\bigr). \]
Its image is the projective closure of the affine curve given
inside $E \times E$ by
\[ (x_1 - 3 b) (x_2 - 3 b) = 81 a^2. \]
The image~$\calC'$ of~$\calC$ under the composition of morphisms
\[ \calC \injects E \times E \subseteq \PP^2 \times \PP^2 \stackrel{\text{Segre}}{\To} \PP^8 \]
has degree~$12$.

We need a bound on the height $h_2(\calC')$. Setting $\xi_j = (x_j - 3b)/(9a)$,
we have $\xi_1 \xi_2 = 1$. Taking $\xi_1 = \zeta$ and $\xi_2 = \zeta^{-1}$,
where $\zeta$ is a root of unity, we get $x_1 = 9 a \zeta + 3b$, $x_2 = 9 a \zeta^{-1} + 3b$,
and $y_1$, $y_2$ are square roots of $(27 a)^2 (a \zeta^{\pm 3} + b \zeta^{\pm 2} + b \zeta^{\pm 1} + a)$.
Using that $a$ and $b$ are rational integers, which implies that the
contributions to the height coming from non-archimedean places vanish,
and the triangle inequality to bound the contributions from the archimedean places
shows that there are infinitely many points $P = (P_1, P_2)$ on the image
of~$\calC$ in $E \times E$ such that
\[ h_2(P) = h_2(P_1) + h_2(P_2) \le \log(1456 a^2(|a|+|b|) + (9|a|+3|b|)^2 + 1). \]
So by Zhang's inequality, we find that
\begin{equation} \label{E:h2Cbound}
  h_2(\calC') \le 24 \log(1456 a^2(|a|+|b|) + (9|a|+3|b|)^2 + 1)
              \le 24 \log 3057 + 72 \log m,
\end{equation}
where $m = \max\{|a|, |b|\}$.

\begin{corollary} \label{C:g2bound}
  Let $\calC \colon y^2 = a x^6 + b x^4 + b x^2 + a$ with $a, b \in \Z$, $a \neq 0, -b, b/3$,
  and let $E$ be as above. Assume that $E(\Q)$ has rank~$1$, and
  let $P_0 \in E(\Q)$ generate the free part of~$E(\Q)$. For a point $P \in \calC(\Q)$, write
  $\pi_1(P) = n_1 P_0 + T_1$, $\pi_2(P) = n_2 P_0 + T_2$
  with $n_1, n_2 \in \Z$ and $T_1, T_2 \in E(\Q)_\tors$. Then
  \begin{align*}
    \min\{|n_1|, |n_2|\} &\le \sqrt{\frac{433.506 h_2(\calC') + 31311.3 + 20808.3 c_1(E) + 2 c_2(E)}{\hat{h}(P_0)}} \\
                         &\le \sqrt{\frac{358956.08 + 93638.80 \log m}{\hat{h}(P_0)}},
  \end{align*}
  where $m = \max\{|a|, |b|\}$.
\end{corollary}

\begin{proof}
  From Theorem~\ref{MAINT} and $\deg(\calC') = 12$, we obtain the bound
  \[ \hat{h}(P) = \hat{h}(\pi_1(P)) + \hat{h}(\pi_2(P))
                \le 72.251 \bigl(12 h_2(\calC') + 144 (6.019 + 4 c_1(E))\bigr) + 4 c_2(E)
  \]
  for points $P \in \calC(\Q)$, where $c_1(E)$, $c_2(E)$ are as in Table~\ref{table1}.
  Since $h_{\calW}(E) \le \tfrac{1}{2} \log 108 + \log m$, we have
  \[ c_1(E) \le 11.733 + 3 \log m \quad\text{and}\quad c_2(E) \le 5.939 + \tfrac{3}{2} \log m, \]
  which using~\eqref{E:h2Cbound} gives 
  \[ \hat{h}(P) \le 717912.16 + 187277.60 \log m. \]
  Also, $\hat{h}(P) = (n_1^2 + n_2^2) \hat{h}(P_0)$,
  so $\min\{|n_1|, |n_2|\} \le \sqrt{\hat{h}(P)/(2 \hat{h}(P_0))}$, which together
  with the bound for~$\hat{h}(P)$ gives the statement.
\end{proof}

The bound in the theorem was chosen to be in a simple form.
In concrete cases, one will use the more precise bound in terms of
$a$ and~$b$ in~\eqref{E:h2Cbound} and also better bounds on $c_1(E)$
and~$c_2(E)$.

We compare this with the bound obtained in~\cite{Kul3}. There curves
with $a = 1$ are studied, where $b$ (denoted~$t$ in~\cite{Kul3})
can be rational. Then (if $E(\Q)$ has rank~$1$) they show that for
all $P \in \calC(\Q)$
\[ h_W(x(P)) \le \tfrac{7}{2} h(b) + \tfrac{1}{2} \log 81468 \le \tfrac{7}{2} h(b) + 5.654. \]
Since the $x$-coordinates of the images of~$P$ on~$E$ are given by
$9 x(P)^{\pm 2} + 3 b$, this translates into
\begin{equation} \label{E:KuleszBound}
  \min\{|n_1|, |n_2|\} \le \sqrt{\frac{12 h(b) + 22.946 + 3 c_3(E)}{\hat{h}(P_0)}}.
\end{equation}
This is considerably smaller than the bound given in Corollary~\ref{C:g2bound}.

\begin{example}
  For a concrete example, consider the curve with $a = b = 1$:
  \[ \calC \colon y^2 = x^6 + x^4 + x^2 + 1. \]
  Then $E$ is the curve 128a1 in the Cremona database~\cite{Cremona}
  (and 128.a2 in~\cite{lmfdb}), and
  $E(\Q) \cong \Z/2\Z \times \Z$. We have $\hat{h}(P_0) > 0.6485$.
  The bound in the theorem above (using the bound for $h_2(\calC')$ in~\eqref{E:h2Cbound}
  and the bounds for $c_1(E)$ and~$c_2(E)$ from Table~\ref{table1}) gives
  \[ \min\{|n_1|, |n_2|\} \le 728. \]
  For comparison, the bound in~\eqref{E:KuleszBound} gives
  \[ \min\{|n_1|, |n_2|\} \le 7. \]
  From this, it is easy to find the set of rational points on~$\calC$:
  \[ \calC(\Q) = \{\infty_+, \infty_-, (-1, \pm 2), (0, \pm 1), (1, \pm 2)\}. \]

  For an example with a larger~$b$, consider $b = 1003$ (this is the
  smallest $b \ge 1000$ such that $E(\Q)$ has rank~$1$).
  Corollary~\ref{C:g2bound} gives a bound of~$354$ for the minimum
  of $|n_1|$ and~$|n_2|$, whereas \eqref{E:KuleszBound} gives a bound of~$4$.

  The fairly large discrepancy (roughly a factor~$100$ for the bound on
  $n_1$ and~$n_2$ and a factor~$10^4$ for the bound on the height) between
  the bounds obtained by the method of the main paper and by Demjanenko's
  method suggests that it might be possible to obtain better bounds from
  the approach taken by Checcoli, Veneziano and Viada than given in Theorem~\ref{MAINT}.
  In any case, the comparison
  in this specific case is perhaps a bit unfair, since the setting is
  rather advantageous for an application of Demjanenko's method.
\end{example}


\subsection{A lower bound for non-integral points} \label{App:nonintegral} \strut

Let $E$ be an elliptic curve over~$\Q$ of rank~$1$ given by a
Weierstrass equation with integral coefficients.
In this section, we consider a curve $\calC \subseteq E \times E$ that is given
by an affine equation of the form
\[ F_1(x_1,y_1) = F_2(x_2,y_2) \]
(where $(x_1,y_1)$ are the affine coordinates on the first
and $(x_2,y_2)$ on the second factor~$E$)
with polynomials $F_1, F_2 \in \Z[x,y]$. Using the equation of~$E$,
we can assume that $F_j(x,y) = f_j(x) + g_j(x) y$ with univariate
polynomials $f_j, g_j \in \Z[x]$. Note that $F_j$ is a rational function
on~$E$ whose only pole is at the origin~$O$ and that
$d_j := \deg F_j = \max\{2 \deg f_j, 3 + 2 \deg g_j\}$.
The \emph{leading coefficient} of~$F_j$ is the coefficient of the
term of largest degree present in~$F_j$.
We also require in the following that $d_1$ is strictly greater than~$d_2$.
Our goal in this section is to obtain a \emph{lower} bound
on the height of a point $P \in \calC(\Q)$.

Let $\ell$ be a prime number. For our purposes the \emph{kernel
of reduction} $K_\ell(E)$ of~$E$ at~$\ell$ is the subgroup
of~$E(\Q_\ell)$ consisting of points reducing mod~$\ell$ to the origin
on the model of~$E$ defined by the given equation. (This may
differ from the more usual notion, which refers to a minimal model
of~$E$, when $E$ has bad reduction at~$\ell$.)
We write $v_\ell$ for the (additive) $\ell$-adic valuation on~$\Q_\ell$,
normalised so that $v_\ell(\ell) = 1$.

We let $t := x/y$ be the standard uniformiser of~$E$ at~$O$.
Then if a point $P \in E(\Q_\ell)$ is in the kernel of reduction,
we have $v_\ell(t(P)) > 0$, and standard properties of formal
groups imply when $\ell$ is odd or when $\ell = 2$
and $E$ is given by an integral Weierstrass equation without
`mixed terms' $y$ or~$xy$ that
\begin{equation} \label{E:formgp}
  v_\ell(t(nP)) = v_\ell(t(P)) + v_\ell(n).
\end{equation}

Let $S$ be a finite set of primes containing the primes dividing
the leading coefficients of $F_1$ and~$F_2$ and also the prime~$2$
if the equation defining~$E$ contains mixed terms.
Then for a prime $\ell \notin S$ and a point
$P \in E(\Q_\ell)$, we have that
\begin{equation} \label{E:pole}
  P \in K_\ell(E) \iff v_\ell(F_j(P)) < 0,
\end{equation}
and in this case we have the relation
\begin{equation} \label{E:vF}
  v_\ell(F_j(P)) = -d_j v_\ell(t(P)).
\end{equation}
We denote the ring of $S$-integers by~$\Z_S$.

\begin{theorem} \label{T:lowerbound}
  Consider $E$, $\calC$ and~$S$ as above (with $d_1 > d_2$). Set
  \[ \lambda = \hat{h}(P_0) \min \{a_\ell^2 \ell^{2\lceil d_1/d_2 \rceil - 2} : \ell \notin S\}, \]
  where $P_0$ is a generator of the free part of~$E(\Q)$ and
  $a_\ell$ is the smallest positive integer such that $a_\ell P_0 \in K_\ell(E) + E(\Q)_\tors$.
  Then
  \[ \calC(\Q) \subseteq \{(O,O)\} \cup \bigl(E(\Z_S) \times E(\Z_S)\bigr)
                     \cup \{P \in E(\Q) \times E(\Q) : \hat{h}(P) \ge \lambda\}.
  \]
\end{theorem}

\begin{proof}
  Assume $P = (P_1,P_2) \in \calC(\Q)$, but $P \neq (O,O)$ and $P \notin E(\Z_S) \times E(\Z_S)$.
  Since $O$ is the only pole of $F_1$ and~$F_2$, we have $P_1 = O \iff P_2 = O$,
  but this case is excluded. By assumption, one of $P_1$ and~$P_2$ is not $S$-integral.
  If $P_1$ is not $S$-integral, then there is a prime $\ell \notin S$ such that $P_1 \in K_\ell(E)$.
  By~\eqref{E:pole}, this implies that $P_2 \in K_\ell(E)$ as well.
  If $P_2$ is not $S$-integral, the same argument applies.
  So $P_1$ and~$P_2$ are both nontrivial points
  in $K_\ell(E) \cap E(\Q)$. Then by~\eqref{E:vF} we must have
  \[ d_1 v_\ell(t(P_1)) = -v_\ell(F_1(P_1)) = -v_\ell(F_2(P_2)) = d_2 v_\ell(t(P_2)). \]
  Now let $P' \in E(\Q)$ be
  a generator of the intersection $E(\Q) \cap K_\ell(E)$
  (this group is isomorphic to~$\Z$ when $E(\Q)$ has rank~$1$; recall
  that the kernel of reduction does not contain nontrivial elements of finite order
  when $\ell$ is odd; $\ell = 2$ is taken care of by our choice of~$S$).
  We can then write
  $P_1 = n_1 P'$, $P_2 = n_2 P'$ with $n_1, n_2 \in \Z$, and we have
  by~\eqref{E:formgp} that
  \[ v_\ell(t(P')) + v_\ell(n_1) = v_\ell(t(P_1)) \qquad\text{and}\qquad
     v_\ell(t(P')) + v_\ell(n_2) = v_\ell(t(P_2)).
  \]
  Combining this with the relation between $v_\ell(t(P_1))$ and~$v_\ell(t(P_2))$,
  we obtain
  \[ v_\ell(n_2) = v_\ell(t(P_2)) - v_\ell(t(P'))
                 = \frac{d_1 - d_2}{d_2} v_\ell(t(P')) + \frac{d_1}{d_2} v_\ell(n_1)
                 \ge \frac{d_1 - d_2}{d_2},
  \]
  since $v_\ell(n_1) \ge 0$ and $v_\ell(t(P')) \ge 1$.
  It follows that $n_2 \ge \ell^{\lceil d_1/d_2 \rceil - 1}$.
  We have that $P' = \pm a_\ell P_0 + T_\ell$ with $T_\ell \in E(\Q)_\tors$,
  and so
  \[ \hat{h}(P) = \hat{h}(P_1) + \hat{h}(P_2)
                = a_\ell^2 (n_1^2 + n_2^2) \hat{h}(P_0)
                \ge a_\ell^2 \ell^{2\lceil d_1/d_2 \rceil - 2} \hat{h}(P_0)
                \ge \lambda,
  \]
  which was to be shown.
\end{proof}

We can combine these results with the upper bound from Theorem~\ref{MAINT}.
If this upper bound is smaller than~$\lambda$, then it follows that
\[ \calC(\Q) \subseteq \{(O,O)\} \cup \bigl(E(\Z_S) \times E(\Z_S)\bigr). \]
Note that $E(\Z_S)$ is a finite set that can easily be determined
in practice once a generator~$P_0$ of the free part of~$E(\Q)$ is known.

In the following, $\ell_{\min}$ denotes the smallest prime not in~$S$.

One way of applying Theorem~\ref{T:lowerbound}
is to consider families of curves in $E \times E$
such that $\ell_{\min}^{d_1/d_2}$ tends to infinity sufficiently fast
compared to the height and the degree of the curves.
Once the parameter is sufficiently large, it follows that the rational
points of all the curves must be contained in some explicit finite
set, so that one can determine the set of rational points on all
the curves in the family.
We will do this in the next section for the examples
$\calC_n$ and~$\calD_n$ given in Theorem~\ref{curveEsp}.

Given a concrete curve, one can also increase the set~$S$ until
$\lambda$ exceeds the upper bound. This is always possible, since
$\lambda \ge \ell_{\min}^2 \hat{h}(P_0)$. The conclusion is again that all rational
points on the curve other than~$(O,O)$ must be $S$-integral,
which may lead to a simpler way of determining this set.

We also state the following special case.

\begin{theorem} \label{T:trivial}
  Assume that, in the situation of Theorem~\ref{T:lowerbound}, $E(\Q)_\tors = 0$
  and $P_0 \notin E(\Z_\ell)$ for some $\ell \notin S$. Then
  \[ \calC(\Q) \subseteq \{(O,O)\}
                 \cup \{P \in E(\Q) \times E(\Q) : \hat{h}(P) \ge \ell^{2\lceil d_1/d_2 \rceil - 2} \hat{h}(P_0)\}.
  \]
\end{theorem}

\begin{proof}
  In this case, all points $P = (P_1,P_2) \in \calC(\Q)$ have
  $P_1, P_2 \in K_\ell(E)$. The argument in the proof of
  Theorem~\ref{T:lowerbound} then applies to all these points with this fixed~$\ell$
  (here $a_\ell = 1$, since $P_0 \in K_\ell(E)$).
\end{proof}

If the lower bound $\ell^{2\lceil d_1/d_2 \rceil - 2}$ exceeds the upper
bound given by Theorem~\ref{MAINT}, then it immediately follows that
the only rational point on~$\calC$ is~$(O,O)$.


\subsection{The curves $\calC_n$ and~$\calD_n$} \label{S:examples} \strut

We recall the examples given in Theorem~\ref{curveEsp}.
The first family of examples consists of
the curves~$\calC_n(E)$ defined as the closure of the subset
of $(E \setminus \{O\})^2$ given by the equation $x_1^n = y_2$, for $n \ge 1$ and
the five elliptic curves~$E = E_1, \ldots, E_5$ as defined in the introduction.
The second family consists of the curves $\calD_n(E_i)$ given
by $\Phi_n(x_1) = y_2$, where $\Phi_n$ is the $n$th cyclotomic
polynomial, for the same set of elliptic curves~$E_i$.

In Theorem~\ref{curveEsp} the sets of
rational points $\calC_n(E_i)(\Q)$ and~$\calD_n(E_i)(\Q)$ are determined
for varying ranges of~$n$.
We will use our results to find $\calC_n(E_i)(\Q)$ and~$\calD_n(E_i)(\Q)$
for \emph{all}~$n$. We recall the upper bounds on~$\hat{h}(P)$
for $P \in \calC_n(E_i)(\Q)$ from
Theorem~\ref{teoEsempi}:
\begin{align*}
  E_1 \colon \hat{h}(P) &\le b_1(n) = 73027 n^2 + 219081 n + 164320 \\
  E_2 \colon \hat{h}(P) &\le b_2(n) = 311345 n^2 + 934033 n + 700566 \\
  E_3 \colon \hat{h}(P) &\le b_3(n) = 373925 n^2 + 1121775 n + 841382 \\
  E_4 \colon \hat{h}(P) &\le b_4(n) = 534732 n^2 + 1604195 n + 1203216 \\
  E_5 \colon \hat{h}(P) &\le b_5(n) = 566995 n^2 + 1700984 n + 1275813
\end{align*}
From Corollary~\ref{CoroDn} we obtain the following bounds
for $P \in \calD_n(E_i)(\Q)$:
\begin{align*}
  E_1 \colon \hat{h}(P) &\le b'_1(n) = (901.5 \cdot 2^{\omega_2(n)} + 18257)(2 \varphi(n) + 3)^2 + 9.7 \\
  E_2 \colon \hat{h}(P) &\le b'_2(n) = (901.5 \cdot 2^{\omega_2(n)} + 77837)(2 \varphi(n) + 3)^2 + 41.4 \\
  E_3 \colon \hat{h}(P) &\le b'_3(n) = (901.5 \cdot 2^{\omega_2(n)} + 93482)(2 \varphi(n) + 3)^2 + 50 \\
  E_4 \colon \hat{h}(P) &\le b'_4(n) = (901.5 \cdot 2^{\omega_2(n)} + 133683)(2 \varphi(n) + 3)^2 + 70 \\
  E_5 \colon \hat{h}(P) &\le b'_5(n) = (901.5 \cdot 2^{\omega_2(n)} + 141749)(2 \varphi(n) + 3)^2 + 75
\end{align*}

In all cases, we can take $S = \emptyset$ in Theorem~\ref{T:lowerbound},
since the leading coefficients of $F_1(x,y) = x^n$ or $\Phi_n(x)$
and $F_2(x,y) = y$ are both~$1$
and the curves are given by short integral Weierstrass equations.
We have $d_1 = 2n$ (respectively, $d_1 = 2 \varphi(n)$) and $d_2 = 3$,
so for $n \ge 2$ (respectively, $n \ge 3$),
the assumption $d_1 > d_2$ is satisfied.

We first consider $E_1$. Let $P_0 = (1,1)$; this is a generator of~$E_1(\Q)$.
Since $P_0$, $2P_0$ and~$3P_0$ are all integral, we have $a_\ell \ge 4$
for all~$\ell$ (and indeed $a_2 = 4$). So we have
\[ \lambda(n) = 16 \cdot 2^{2\lceil 2n/3 \rceil - 2} \hat{h}(P_0) \ge 2^{4n/3 + 2} \hat{h}(P_0). \]
This is larger than~$b_1(n)$ as soon as $n \ge 19$.
For $\calD_n(E_1)$,
we have to compare $\lambda(\varphi(n))$ with~$b'_1(n)$. We use the crude
bound $2^{\omega_2(n)} \le \varphi(n)$; we then have that
$\lambda(\varphi(n)) \ge b'_1(n)$ for $\varphi(n) \ge 19$, which covers
all $n \ge 61$.
So for $n \ge 19$, we get from Theorem~\ref{T:lowerbound} that
\[ \calC_n(E_1)(\Q) \subseteq \{(O,O)\} \cup \bigl(E_1(\Z) \times E_1(\Z)\bigr) \]
and for $n \ge 61$, we get that
\[ \calD_n(E_1)(\Q) \subseteq \{(O,O)\} \cup \bigl(E_1(\Z) \times E_1(\Z)\bigr). \]
We have that
\[ E_1(\Z) = \{(1, \pm 1), (2, \pm 3)\, (13, \pm 47)\} = \{\pm P_0, \pm 2P_0, \pm 3P_0\} \]
(as obtained by a quick computation in Magma~\cite{Magma}, for example).

To deal with~$\calC_n(E_1)$, we now only
have to check which pairs of such points can satisfy the relation $x_1^n = y_2$.
The only possibilities are $y_2 = 1$, so $P_2 = (1,1)$ and $x_1 = 1$, so
$P_1 = (1, \pm 1)$. Since the cases $n < 19$
are covered by Theorem~\ref{curveEsp}, we obtain the following result.

\begin{corollary}
  For all $n \ge 1$, we have
  \[ \calC_n(E_1)(\Q) = \bigl\{(O,O), \bigl((1,1),(1,1)\bigr), \bigl((1,-1),(1,1)\bigr)\bigr\}. \]
\end{corollary}

We now consider $\calD_n(E_1)$. We have to solve the equation $\Phi_n(x_1) = y_2$,
with $x_1 \in  \{1, 2, 13\}$ and $y_2 \in \{\pm 1, \pm 3, \pm 47\}$.
The easy estimate $|\Phi_n(2)| > 5^{\varphi(n)/4}$ and the even easier
estimate $|\Phi_n(13)| \ge 12^{\varphi(n)}$ show that $x_1 = 1$ is the
only possibility (when $n \ge 61$). We have the well-known fact that $\Phi_n(1) = 1$
unless $n = 1$ or $n$ is a prime power, and $\Phi_{p^m}(1) = p$.
This proves the following statement for $n \ge 61$; the remaining
cases with $n \ge 7$ are covered by Theorem~\ref{curveEsp}, which also shows that for $n \le 6$,
it is still true that all rational points other than $(O, O)$ on~$\calD_n(E_1)$
are pairs of integral points on~$E_1$, but there are some deviations
from the pattern in the statement below (coming from small values of~$\Phi_n(2)$:
$\Phi_1(2) = 1$, $\Phi_2(2) = \Phi_6(2) = 3$).

\begin{corollary}
  For all $n \ge 7$,
  \begin{align*}
    \calD_n(E_1)(\Q) &= \{(O, O), \bigl((1, 1), (1, 1)\bigr), \bigl((1, -1), (1, 1)\bigr)\}
                     & & \text{if $n$ is not a prime power,} \\
    \calD_n(E_1)(\Q) &= \{(O, O)\}
                     & & \text{if $n = p^m$ with $p \neq 3, 47$,} \\
    \calD_n(E_1)(\Q) &= \{(O, O), \bigl((1, 1), (2, 3)\bigr), \bigl((1, -1), (2, 3)\bigr)\}
                     & & \text{if $n = 3^m$,} \\
    \calD_n(E_1)(\Q) &= \{(O, O), \bigl((1, 1), (13, 47)\bigr), \bigl((1, -1), (13, 47)\bigr)\}
                     & & \text{if $n = 47^m$.}
  \end{align*}
\end{corollary}

Now we consider the remaining curves~$E_i$, $i = 2,3,4,5$.
In each case $E_i(\Q) \cong \Z$, and the generator
is not $\ell$-adically integral for $\ell = 491$, $11$, $1418579$, and~$3956941$,
when $i = 2$, $3$, $4$ and~$5$, respectively. So we can apply Theorem~\ref{T:trivial}
with this~$\ell$. The lower bound exceeds the upper bound~$b_i(n)$ (respectively, $b'_i(n)$)
when $n \ge 3$ (respectively, $n \ge 7$) for $i = 2$,
when $n \ge 6$ (respectively, $n \ge 19$) for $i = 3$
and when $n \ge 3$ (respectively, $n \ge 7$) for $i = 4$ and~$i = 5$. So for these ranges,
we obtain immediately that $\calC_n(E_i)(\Q) = \calD_n(E_i)(\Q) = \{(O,O)\}$.
The remaining cases are taken care of by Theorem~\ref{curveEsp}; therefore
we have now proved the following.

\begin{corollary}
  For all $n \ge 1$ and $i = 2,3,4,5$, we have
  \[ \calC_n(E_i)(\Q) = \calD_n(E_i)(\Q) = \{(O,O)\}. \]
\end{corollary}

\medskip
\noindent\small
Mathematisches Institut, Universit\"at Bayreuth, 95440 Bayreuth, Germany. \\
Email: Michael.Stoll@uni-bayreuth.de. \\
WWW: \href{http://www.computeralgebra.uni-bayreuth.de}{http://www.computeralgebra.uni-bayreuth.de}.




\begin{bibdiv}
\begin{biblist}



\bib{BG06}{book}{
   author={Bombieri, E.},
   author={Gubler, W.},
   title={Heights in Diophantine geometry},
   series={New Mathematical Monographs},
   volume={4},
   publisher={Cambridge University Press, Cambridge},
   date={2006},
   pages={xvi+652},
   isbn={978-0-521-84615-8},
   isbn={0-521-84615-3},
   review={\MR{2216774}},
   doi={10.1017/CBO9780511542879},
}

\bib{BMZ1}{article}{
   author={Bombieri, E.},
   author={Masser, D.},
   author={Zannier, U.},
   title={Anomalous subvarieties---structure theorems and applications},
   journal={Int. Math. Res. Not. IMRN},
   date={2007},
   number={19},
   pages={Art. ID rnm057, 33},
   issn={1073-7928},
   review={\MR{2359537}},
   doi={10.1093/imrn/rnm057},
}

\bib{Magma}{article}{
   author={Bosma, W.},
   author={Cannon, J.},
   author={Playoust, C.},
   title={The Magma algebra system. I. The user language},
   journal={J. Symbolic Comput.},
   volume={24},
   date={1997},
   number={3-4},
   pages={235--265},
   url={See also the Magma home page at http://magma.maths.usyd.edu.au/magma/},
}

\bib{BGSGreen}{article}{
   author={Bost, J.-B.},
   author={Gillet, H.},
   author={Soul{\'e}, C.},
   title={Heights of projective varieties and positive Green forms},
   journal={J. Amer. Math. Soc.},
   volume={7},
   date={1994},
   number={4},
   pages={903--1027},
   issn={0894-0347},
   review={\MR{1260106}},
   doi={10.2307/2152736},
}

\bib{StollMagma}{article}{
   author={Bruin, N.},
   author={Stoll, M.},
   title={The Mordell-Weil sieve: proving non-existence of rational points on curves},
   journal={LMS J. Comput. Math.},
   volume={13},
   date={2010},
   pages={272--306},
   issn={1461-1570},
   review={\MR{2685127}},
   doi={10.1112/S1461157009000187},
}

\bib{Chab}{article}{
   author={Chabauty, C.},
   title={Sur les points rationnels des courbes alg\'ebriques de genre sup\'erieur \`a l'unit\'e},
   language={French},
   journal={C. R. Acad. Sci. Paris},
   volume={212},
   date={1941},
   pages={882--885},
   review={\MR{0004484}},
}

\bib{ExpTAC}{article}{
  author={Checcoli, S.},
  author={Veneziano, F.},
  author={Viada, E.},
  title={On the explicit Torsion Anomalous Conjecture},
  journal={to appear in Trans. Amer. Math. Soc.},
}

\bib{Coleman}{article}{
   author={Coleman, R. F.},
   title={Effective Chabauty},
   journal={Duke Math. J.},
   volume={52},
   date={1985},
   number={3},
   pages={765--770},
   issn={0012-7094},
   review={\MR{808103}},
   doi={10.1215/S0012-7094-85-05240-8},
}

\bib{Cremona}{misc}{
  author={Cremona, J. E.},
  title={Elliptic Curve Data},
  note={\url{http://johncremona.github.io/ecdata/}},
  date={2015},
  url={http://johncremona.github.io/ecdata/},
}

\bib{Demj}{article}{
   author={Demjanenko, V. A.},
   title={Rational points of a class of algebraic curves},
   language={Russian},
   journal={Izv. Akad. Nauk SSSR Ser. Mat.},
   volume={30},
   date={1966},
   pages={1373--1396},
   issn={0373-2436},
   review={\MR{0205991}},
   translation={title={Rational points on a class of algebraic curves},
                journal={Amer. Math. Soc. Transl.},
                volume={66},
                date={1968},
                pages={246--272},
               },
}

\bib{FaltingsTeo}{article}{
   author={Faltings, G.},
   title={Endlichkeitss\"atze f\"ur abelsche Variet\"aten \"uber Zahlk\"orpern},
   language={German},
   journal={Invent. Math.},
   volume={73},
   date={1983},
   number={3},
   pages={349--366},
   issn={0020-9910},
   review={\MR{718935}},
   doi={10.1007/BF01388432},
}

\bib{FaltingsAnnals}{article}{
   author={Faltings, G.},
   title={Diophantine approximation on abelian varieties},
   journal={Ann. of Math. (2)},
   volume={133},
   date={1991},
   number={3},
   pages={549--576},
   issn={0003-486X},
   review={\MR{1109353}},
   doi={10.2307/2944319},
}

\bib{FaltingsBarsotti}{article}{
   author={Faltings, G.},
   title={The general case of S. Lang's conjecture},
   conference={
      title={Barsotti Symposium in Algebraic Geometry},
      address={Abano Terme},
      date={1991},
   },
   book={
      series={Perspect. Math.},
      volume={15},
      publisher={Academic Press, San Diego, CA},
   },
   date={1994},
   pages={175--182},
   review={\MR{1307396}},
}

\bib{Flynn}{article}{
   author={Flynn, E. V.},
   title={A flexible method for applying Chabauty's theorem},
   journal={Compositio Math.},
   volume={105},
   date={1997},
   number={1},
   pages={79--94},
   issn={0010-437X},
   review={\MR{1436746}},
   doi={10.1023/A:1000111601294},
}

\bib{Fulton}{book}{
   author={Fulton, W.},
   title={Intersection theory},
   series={Ergebnisse der Mathematik und ihrer Grenzgebiete (3) [Results in Mathematics and Related Areas (3)]},
   volume={2},
   publisher={Springer-Verlag, Berlin},
   date={1984},
   pages={xi+470},
   isbn={3-540-12176-5},
   review={\MR{732620}},
   doi={10.1007/978-3-662-02421-8},
}

\bib{Kul2}{article}{
   author={Girard, M.},
   author={Kulesz, L.},
   title={Computation of sets of rational points of genus-3 curves via the Dem\cprime janenko-Manin method},
   journal={LMS J. Comput. Math.},
   volume={8},
   date={2005},
   pages={267--300},
   issn={1461-1570},
   review={\MR{2193214}},
   doi={10.1112/S1461157000000991},
}

\bib{Hab08}{article}{
   author={Habegger, P.},
   title={Intersecting subvarieties of ${\bf G}\sp n\sb m$ with algebraic subgroups},
   journal={Math. Ann.},
   volume={342},
   date={2008},
   number={2},
   pages={449--466},
   issn={0025-5831},
   review={\MR{2425150}},
   doi={10.1007/s00208-008-0242-3},
}

\bib{HindryAutour}{article}{
   author={Hindry, M.},
   title={Autour d'une conjecture de Serge Lang},
   language={French},
   journal={Invent. Math.},
   volume={94},
   date={1988},
   number={3},
   pages={575--603},
   issn={0020-9910},
   review={\MR{969244}},
   doi={10.1007/BF01394276},
}

\bib{HS}{book}{
   author={Hindry, M.},
   author={Silverman, J. H.},
   title={Diophantine geometry},
   series={Graduate Texts in Mathematics},
   volume={201},
   note={An introduction},
   publisher={Springer-Verlag, New York},
   date={2000},
   pages={xiv+558},
   isbn={0-387-98975-7},
   isbn={0-387-98981-1},
   review={\MR{1745599}},
   doi={10.1007/978-1-4612-1210-2},
}

\bib{BiquadraticIrreducibility}{article}{
   author={Kappe, L.-C.},
   author={Warren, B.},
   title={An elementary test for the Galois group of a quartic polynomial},
   journal={Amer. Math. Monthly},
   volume={96},
   date={1989},
   number={2},
   pages={133--137},
   issn={0002-9890},
   review={\MR{992075}},
   doi={10.2307/2323198},
}

\bib{KuleszApplicazioneMD}{article}{
   author={Kulesz, L.},
   title={Application de la m\'ethode de Dem\cprime janenko-Manin \`a certaines familles de courbes de genre 2 et 3},
   language={French, with English and French summaries},
   journal={J. Number Theory},
   volume={76},
   date={1999},
   number={1},
   pages={130--146},
   issn={0022-314X},
   review={\MR{1688176}},
   doi={10.1006/jnth.1998.2339},
}

\bib{Kul3}{article}{
   author={Kulesz, L.},
   author={Matera, G.},
   author={Schost, E.},
   title={Uniform bounds on the number of rational points of a family of curves of genus 2},
   journal={J. Number Theory},
   volume={108},
   date={2004},
   number={2},
   pages={241--267},
   issn={0022-314X},
   review={\MR{2098638}},
   doi={10.1016/j.jnt.2004.05.013},
}

\bib{lmfdb}{misc}{
  label={LMFDB},
  title={L-functions and modular forms database},
  note={\url{http://www.lmfdb.org/EllipticCurve/Q}},
}

\bib{Manin}{article}{
   author={Manin, Ju. I.},
   title={The $p$-torsion of elliptic curves is uniformly bounded},
   language={Russian},
   journal={Izv. Akad. Nauk SSSR Ser. Mat.},
   volume={33},
   date={1969},
   pages={459--465},
   issn={0373-2436},
   review={\MR{0272786}},
}

\bib{MW90}{article}{
   author={Masser, D. W.},
   author={W{\"u}stholz, G.},
   title={Estimating isogenies on elliptic curves},
   journal={Invent. Math.},
   volume={100},
   date={1990},
   number={1},
   pages={1--24},
   issn={0020-9910},
   review={\MR{1037140}},
   doi={10.1007/BF01231178},
}

\bib{MP}{article}{
   author={McCallum, W.},
   author={Poonen, B.},
   title={The method of Chabauty and Coleman},
   language={English, with English and French summaries},
   conference={
      title={Explicit methods in number theory},
   },
   book={
      series={Panor. Synth\`eses},
      volume={36},
      publisher={Soc. Math. France, Paris},
   },
   date={2012},
   pages={99--117},
   review={\MR{3098132}},
}

\bib{Mazur}{article}{
   author={Mazur, B.},
   title={Modular curves and the Eisenstein ideal},
   journal={Inst. Hautes \'Etudes Sci. Publ. Math.},
   number={47},
   date={1977},
   pages={33--186 (1978)},
   issn={0073-8301},
   review={\MR{488287}},
}

\bib{MordellConj}{article}{
  author={Mordell, L. J.},
  title={On the rational solutions of the indeterminate equation of the third and fourth degrees},
  journal={Proc. Cambridge Philos. Soc.},
  volume={21},
  date={1922},
  pages={179--192},
}

\bib{Parent99}{article}{
   author={Parent, P.},
   title={Bornes effectives pour la torsion des courbes elliptiques sur les corps de nombres},
   language={French, with French summary},
   journal={J. Reine Angew. Math.},
   volume={506},
   date={1999},
   pages={85--116},
   issn={0075-4102},
   review={\MR{1665681}},
   doi={10.1515/crll.1999.009},
}

\bib{PARI}{misc}{
  label={PARI},
  author={The PARI~Group},
  title={PARI/GP version {\tt 2.8.0}},
  date={2015},
  note={\url{http://pari.math.u-bordeaux.fr/}},
}

\bib{patriceI}{article}{
   author={Philippon, P.},
   title={Sur des hauteurs alternatives. I},
   language={French},
   journal={Math. Ann.},
   volume={289},
   date={1991},
   number={2},
   pages={255--283},
   issn={0025-5831},
   review={\MR{1092175}},
   doi={10.1007/BF01446571},
}

\bib{patrice}{article}{
   author={Philippon, P.},
   title={Sur des hauteurs alternatives. III},
   language={French},
   journal={J. Math. Pures Appl. (9)},
   volume={74},
   date={1995},
   number={4},
   pages={345--365},
   issn={0021-7824},
   review={\MR{1341770}},
}

\bib{preprintPhilippon}{misc}{
  author={Philippon, P.},
  title={Sur une question d'orthogonalit{\'e} dans les puissances de courbes elliptiques},
  date={2012},
  note={Preprint, hal--00801376, \url{https://hal.archives-ouvertes.fr/hal-00801376/document}},
}

\bib{gaelLang}{article}{
   author={R{\'e}mond, G.},
   title={D\'ecompte dans une conjecture de Lang},
   language={French, with English summary},
   journal={Invent. Math.},
   volume={142},
   date={2000},
   number={3},
   pages={513--545},
   issn={0020-9910},
   review={\MR{1804159}},
   doi={10.1007/s002220000095},
}

\bib{serreMWThm}{book}{
   author={Serre, J.-P.},
   title={Lectures on the Mordell-Weil theorem},
   series={Aspects of Mathematics, E15},
   note={Translated from the French and edited by Martin Brown from notes by Michel Waldschmidt},
   publisher={Friedr. Vieweg \& Sohn, Braunschweig},
   date={1989},
   pages={x+218},
   isbn={3-528-08968-7},
   review={\MR{1002324}},
   doi={10.1007/978-3-663-14060-3},
}

\bib{Siksek}{article}{
   author={Siksek, S.},
   title={Explicit Chabauty over number fields},
   journal={Algebra Number Theory},
   volume={7},
   date={2013},
   number={4},
   pages={765--793},
   issn={1937-0652},
   review={\MR{3095226}},
   doi={10.2140/ant.2013.7.765},
}


\bib{SilvermanArithmeticEllipticCurves}{book}{
   author={Silverman, J. H.},
   title={The arithmetic of elliptic curves},
   series={Graduate Texts in Mathematics},
   volume={106},
   publisher={Springer-Verlag, New York},
   date={1986},
   pages={xii+400},
   isbn={0-387-96203-4},
   review={\MR{817210}},
   doi={10.1007/978-1-4757-1920-8},
}

\bib{Silverman1987}{article}{
   author={Silverman, J. H.},
   title={Rational points on certain families of curves of genus at least $2$},
   journal={Proc. London Math. Soc. (3)},
   volume={55},
   date={1987},
   number={3},
   pages={465--481},
   issn={0024-6115},
   review={\MR{907229}},
   doi={10.1112/plms/s3-55.3.465},
}

\bib{SilvermanDifferenceHeights}{article}{
   author={Silverman, J. H.},
   title={The difference between the Weil height and the canonical height on elliptic curves},
   journal={Math. Comp.},
   volume={55},
   date={1990},
   number={192},
   pages={723--743},
   issn={0025-5718},
   review={\MR{1035944}},
   doi={10.2307/2008444},
}

\bib{Sil}{article}{
   author={Silverman, J. H.},
   title={Computing rational points on rank $1$ elliptic curves via $L$-series and canonical heights},
   journal={Math. Comp.},
   volume={68},
   date={1999},
   number={226},
   pages={835--858},
   issn={0025-5718},
   review={\MR{1627825}},
   doi={10.1090/S0025-5718-99-01068-6},
}

\bib{Stoll}{article}{
   author={Stoll, M.},
   title={Rational points on curves},
   language={English, with English and French summaries},
   journal={J. Th\'eor. Nombres Bordeaux},
   volume={23},
   date={2011},
   number={1},
   pages={257--277},
   issn={1246-7405},
   review={\MR{2780629}},
}

\bib{ioannali}{article}{
   author={Viada, E.},
   title={The intersection of a curve with algebraic subgroups in a product of elliptic curves},
   journal={Ann. Sc. Norm. Super. Pisa Cl. Sci. (5)},
   volume={2},
   date={2003},
   number={1},
   pages={47--75},
   issn={0391-173X},
   review={\MR{1990974}},
}

\bib{via15}{article}{
  author={Viada, E.},
  title={Explicit height bounds and the effective Mordell-Lang Conjecture},
  journal={Rivista di Matematica della Universit\`a di Parma},
  note={Proceedings of the ``Third Italian Number Theory Meeting" Pisa (Italy), September 21-24, 2015},
volume={7},
date={2016},
number={1},
pages={101--131},
}

\bib{Vojta91}{article}{
   author={Vojta, P.},
   title={Siegel's theorem in the compact case},
   journal={Ann. of Math. (2)},
   volume={133},
   date={1991},
   number={3},
   pages={509--548},
   issn={0003-486X},
   review={\MR{1109352}},
   doi={10.2307/2944318},
}

\bib{Zan12}{book}{
   author={Zannier, U.},
   title={Some problems of unlikely intersections in arithmetic and geometry},
   series={Annals of Mathematics Studies},
   volume={181},
   note={With appendixes by David Masser},
   publisher={Princeton University Press, Princeton, NJ},
   date={2012},
   pages={xiv+160},
   isbn={978-0-691-15371-1},
   review={\MR{2918151}},
}

\bib{Zhang95}{article}{
   author={Zhang, S.},
   title={Positive line bundles on arithmetic varieties},
   journal={J. Amer. Math. Soc.},
   volume={8},
   date={1995},
   number={1},
   pages={187--221},
   issn={0894-0347},
   review={\MR{1254133}},
   doi={10.2307/2152886},
}

\bib{ZimmerAltezze}{article}{
   author={Zimmer, H. G.},
   title={On the difference of the Weil height and the N\'eron-Tate height},
   journal={Math. Z.},
   volume={147},
   date={1976},
   number={1},
   pages={35--51},
   issn={0025-5874},
   review={\MR{0419455}},
}


\end{biblist}
\end{bibdiv}
\end{document}